\documentclass[english,a4paper,10pt]{amsart}%12pt
\usepackage[latin1]{inputenc}
\usepackage{amssymb,amsthm}
\usepackage[all]{xy}
\usepackage{amsmath}
\usepackage{indentfirst}
\usepackage{amsfonts}
\usepackage{textcomp}
\usepackage{graphicx}
\usepackage{enumerate}
\usepackage{hyperref}

\newtheorem{lemma}{Lemma}[section]
\newtheorem{theorem}[lemma]{Theorem}
\newtheorem{proposition}[lemma]{Proposition}
\newtheorem{corollary}[lemma]{Corollary}
\theoremstyle{definition}
\newtheorem{definition}[lemma]{Definition}
\theoremstyle{remark}
\newtheorem{remark}[lemma]{Remark}

\newcommand{\mult}{\textrm{mult}}

\newcommand{\vol}{\textrm{vol}}

\def\supp{\text{Supp}}

\def\codim{\text{codim}}

\newcommand{\Ox}{\mathcal{O}_X}

\begin{document}

\title{Pluricanonical systems for 3-folds and 4-folds of general type}
\thanks{{\it Math classification:} 14J30, 14J35, 14J40, 14E05}
\thanks{\today}

\author{Lorenzo Di Biagio}
\address{Dipartimento di Matematica, Sapienza -  Universit\`a di Roma - Piazzale Aldo Moro 5, 00185, Roma, Italy}
\email[1]{dibiagio@mat.uniroma1.it}
\email[2]{lorenzo.dibiagio@gmail.com}
\begin{abstract}
We explicitly find lower bounds on the volume of threefolds and fourfolds of general type in order to have non-vanishing of pluricanonical systems and birationality of pluricanonical maps. In the case of threefolds of  large volume, we also give necessary and sufficient conditions for the fourth canonical map to be birational. 
\end{abstract}

\maketitle

\setcounter{section}{0}
\setcounter{lemma}{0} 

\section{Introduction}
As it is well known, one of the guiding problems in algebraic geometry is to classify all algebraic varieties up to birational equivalence. Hence it is natural to study pluricanonical systems and the structure of the related pluricanonical maps, especially for varieties of general type. 

By the definition, if $X$ is a complex projective smooth variety of general type and dimension $d$ then the plurigenera $P_n =h^0(X,nK_X)$ grow like $n^d$ and $|nK_X|$ is birational for $n$ sufficiently large (meaning that the pluricanonical map $\phi_{|nK_X|}: X \dashrightarrow \mathbb{P}(H^0(X, nK_X))$ is birational onto its image). It is then legitimate to wonder if it is possible to find an explicit number $n_d$, potentially the minimal one, such that $n_d$ does not depend on $X$ (but only on $d$) and $P_n \not = 0$ or $|nK_X|$ is birational for all $n \geq n_d$.

For curves and surfaces of general type results of this kind are already known since a long time: by simple applications of Riemann-Roch theorem, for curves we have that $P_n \not = 0$ as soon as $n \geq 1$ and $|nK_X|$ is birational as soon as $n \geq 3$; for surfaces Bombieri proved in 1973 (see \cite{Bombieri}) that $P_n \not = 0$ for $n \geq 2$ and $|nK_X|$ is birational for $n \geq 5$.

For varieties of higher dimension recent advances have been made independently by Hacon--McKernan (see \cite{HcMcK}) and Takayama (see \cite{Takayama}) using ideas of Tsuji. They proved that actually, and for every $d$, this $n_d$ exists, even if their methods do not directly allow us to compute it. In the case of threefolds J.A. Chen and M. Chen proved in \cite{ChenChenII} that $P_n>0$ for every $n \geq 27$ and that $|nK_X|$ is birational for all $n \geq 73$. But if one requires in addition that some invariant of $X$ is big then it is possible to have better effective statements. This is the content of an article of G.T. Todorov (see \cite{Todorov}) who proved that if the volume of $X$ is sufficiently large then $P_2 \not = 0$ and $|5K_X|$ is birational. Note also that these results are optimal in the sense that there exist threefolds of arbitrarily large volume with $P_1=0$ and $|4K_X|$ not birational.

In this work we develop a strategy to effectively study non-vanishing (and size) of pluricanonical systems and birationality of pluricanonical maps for varieties of general type of any dimension and large volume, also with respect to the genus of the curves lying on the variety.  

As a matter of fact we succeed in improving Todorov's results for threefolds (also studying higher plurigenera and higher pluricanonical maps) and in finding effective results even for fourfolds and, partially, for fivefolds. 

We also manage to give characterizations for threefolds of general type with birational fourth pluricanonical map. We just need to assume that the volume is sufficiently large: to the author's knowledge this approach has never been considered before.\\

%In this work, assuming the volume sufficiently large, we improve Todorov's results for threefolds (also studying higher plurigenera and higher pluricanonical maps), we find effective results also for fourfolds and, partially, for fivefolds and, lastly, we find characterizations for threefolds of general type with birational fourth pluricanonical map. 

From this point on, we will give some details about the most significant results that we obtain.

In the case of threefolds of general type we prove the following theorems:
\begin{theorem} (see Theorem \ref{threefold}).
Let $X$ be a smooth projective threefold of general type such that $\vol(X) > \alpha^3$.  If  $\alpha \geq 879$ then $h^0(2K_X) \geq 1$ and if  $\alpha \geq 432(n+1)-3$ then $h^0((n+1)K_X)\geq n$, for all $n \geq 2$.
\end{theorem}

\begin{theorem} (see Theorem \ref{birthreefold}).
Let $X$ be a smooth projective threefold of general type such that $\vol(X) > \alpha^3$. If $\alpha > 1917\sqrt[3]{2}$   then $|lK_X|$ gives a birational map for every $l \geq 5$.
\end{theorem}

In both cases we have much more precise estimates on $\alpha$, depending on $l$ and on the genus of the curves lying on $X$. See Theorem \ref{threefold} and Theorem \ref{birthreefold}, respectively.

 %we explicitly find a function  $\alpha_0$ depending on an integer $n \geq 1$ and on the minimal genus $g$ of the curves that cover $X$ such that if $\vol(X) > \alpha_0(g,n)^3$ then $h^0((n+1)K_X) \geq n$ (see theorem \ref{threefold}); we also explicitly find a function $\alpha_1(g,l)$, with $l \in \mathbb{N}$, $l \geq 5$, such that if $\vol(X) >\alpha_1(g,l)^3$ then $|lK_X|$ is birational (see theorem \ref{birthreefold}). In the case of varieties of dimension $d$ we find analogous results (see theorem \ref{dimensionealta}  and theorem \ref{birationalitydimalta}), but depending on the general lower bounds of the volume of varieties of dimension equal or smaller than $d-1$, bounds that are still not known. 
 
We find analogous results also for fourfolds of general type. Using a lower bound on the volume of threefolds of general type given by J.Chen and M.Chen (see \cite{ChenChenII}) we prove:

\begin{theorem} (see Corollary \ref{explicitnonvanishing}).
Let $X$ be a smooth projective fourfold of general type such that $\vol(X) > \alpha^4$. If $\alpha \geq 1709$ then $h^0(X, (1+m)K_X) \geq n$ for all $n \geq 1$ and all $m \geq 191n$.
\end{theorem}

\begin{theorem} (see Corollary \ref{explicitbirationality}). Let $X$ be a smooth projective fourfold of general type such that $\vol(X) > \alpha^4$. If $\alpha \geq 2816$ then $|l K_X|$ gives a birational map for every $l \geq 817$.
\end{theorem}

As before, we have more precise estimates on $\alpha$, depending on the genus of the curves lying on $X$. See Corollary \ref{explicitnonvanishing} and Corollary \ref{explicitbirationality}, respectively.
%On the contrary, for fourfolds we use a lower bound on the volume of threefolds of general type given by J.Chen and M.Chen (see \cite{ChenChenII}) to find explicit results (see corollary \ref{explicitnonvanishing} and corollary \ref{explicitbirationality}), even if in this case they are probably not optimal (in the sense explained before). 

In the case of varieties of general type of dimension $d$, when $l$ is sufficiently large, we also find functions $\alpha_1(d,l), \alpha_2(d,l)$ such that if $\vol(X)>\alpha_1(d,l)^d$ then either $h^0(lK_X) \not = 0$ or $X$ is birational to a fibre space over a curve such that the general fibre has small volume (see Theorem \ref{fivefoldnonvanishing}) and if $\vol(X)>\alpha_2(d,l)^d$ then either $|lK_X|$ is birational or $X$ is birational to a fibre space over a curve such that the general fibre has small volume (see Theorem \ref{fivefoldbirationality}); both these functions depend on the lower bounds of the volume of varieties of dimension equal or smaller than $d-2$, thus allowing us to find explicit results also in the case of fivefolds.

Another interesting question that arises naturally when dealing with threefolds of general type is to study when $|4K_X|$ is birational. It is clear that $|4K_X|$ cannot be birational if $X$ is birationally equivalent to a fibration over a curve $B$ such that the general fibre is a minimal surface $S$ with $K_S^2=1$ and with geometric genus $=2$, since in this case $|4K_S|$ is not birational. In general the converse does not hold (see Remark \ref{controesempiomappa4canonica}), but it turns out that it actually holds when the volume of $X$ is sufficiently large. We prove the following: 
\begin{theorem} (see Corollary \ref{4thmap}).
Let $X$ be a smooth projective threefold of general type such that $\vol(X) > \alpha^3$. If $\alpha > 6141\sqrt[3]{2}$ %(or, in case $g \geq 42$, $\alpha > 141 \sqrt[3]{2}$)
 then $|4K_X|$ does not give a birational map if, and only if, $X$ is birational to a fibre space $X''$, with $f: X'' \rightarrow B$, where $B$ is a curve, such that the general fiber $X''_b$ is a smooth minimal surface of general type with volume $1$ and geometric genus $p_g=2$.
\end{theorem}
Again, we have better estimates on $\alpha$ depending on the genus of the curves on $X$: see Corollary \ref{4thmap}.
%In this work we prove that in general also the converse holds: we explicitly find a function $\alpha_4(g)$ such that if $\vol(X) > \alpha_4(g)^3$ then $|4K_X|$ is not birational if and only if $X$ is birationally equivalent to a fibration as above (cf. corollary \ref{4thmap}). 

We also prove some analogous results about the birationality of $|3K_X|$ and $|2K_X|$, but in this case we are bound to add some hypotheses (cf. Corollary \ref{3rdmap} and \ref{2ndmap}, respectively). 

The birationality of $|4K_X|$ has already been analyzed also by Lee, Dong, M.Chen, Zhang. Actually both Dong in \cite{Dong} and Chen--Zhang in \cite{ChenZhang}, requiring that the geometric genus (rather than the volume) of $X$ is sufficiently large ($h^0(K_X) \geq 7$ for Dong, $h^0(K_X) \geq 5$ for Chen--Zhang),  give characterizations for the birationality of the fourth pluricanonical map. Note that the largeness of the geometric genus is not implied by the largeness of the volume (see Remark \ref{controesempio}).

\section{Preliminaries} \label{preliminaries}
\subsection{Notation}

We will work over the field of complex numbers, $\mathbb{C}$. %As in \cite{LazI}, \cite{LazII} a {\textit{scheme}} is a separated algebraic scheme of finite type over $\mathbb{C}$. A {\textit{variety}} is a reduced, irreducible scheme. A {\textit{curve}} is a variety of dimension $1$. A {\textit{surface}} is a variety of dimension $2$. 
A $d$-{\textit{fold}} is a variety of dimension $d$. We will usually deal with closed points of schemes, unless otherwise specified.\\

%In our terminology a {\textit{countable set}} is a set that has a bijection over a subset of $\mathbb{N}$, the set of natural numbers. Thus it can be finite or infinite. By $\mathbb{N}^+$ we will denote the set $\mathbb{N} \setminus \{0\}$; by $\mathbb{Q}^+$ the set $\{q \in \mathbb{Q} | q > 0\}$. \\

Unless otherwise specified a \textit{divisor} or a $\mathbb{Q}$-\textit{divisor} is meant to be Weil. A divisor is called $\mathbb{Q}$-\textit{Cartier} if an integral multiple is a Cartier divisor. Of course when we work on smooth varieties Weil and Cartier divisors coincide.\\

Let $q \in \mathbb{Q}$: we write $[q]$, $\{q\}$ for the round-down and fractional part of $q$, respectively. Recall that $[q]$ is the greatest integer $\leq q$ and $\{q\}=q-[q]$. \\

If $X$ is a variety and $D$ a Weil-$\mathbb{Q}$-divisor on $X$, when writing $D=\sum_i q_i D_i$ we will assume that the $D_i$'s are distinct prime divisors. Given the case, we also define the round-down of $D$, $[D]$, as $[D]:=\sum_i [q_i]D_i$.\\

A projective morphism $f: X \rightarrow Y$ is called an (algebraic) {\textit{fibre space}} (according to Mori) if $X$, $Y$ are smooth projective varieties, $f$ is surjective and $f_*(\mathcal{O}_X)=\mathcal{O}_Y$. Notice that, under this definition, $f_*(\mathcal{O}_X)=\mathcal{O}_Y$ is the same as requiring $f$ to have connected fibres.

\subsection{Topological issues}
In this section we will recall some basic definitions and state some easy results of topological flavour that will be used in the proof of the main theorems.

\begin{definition}
Let $X$ be a variety. Let $P \subseteq X$. $P$ is called \textit{very general} if it is the complement of a countable union of proper closed subvarieties of $X$. $P$ is called \textit{countably dense} if it is not contained in the union of countably many proper closed subvarieties of $X$.
\end{definition}

As we will see in the following lemma, countable density is a property stronger than Zariski-density but not as much constraining as being very general. If (very) general sets will usually be the starting point of our analysis it is also true that manipulating these sets leads us to face countably dense sets rather than other (very) general sets. For example if we randomly decompose a very general set into a finite (or countable) union of disjoint sets then we loose information about being very general but we rest assured that at least one of the new sets is countably dense. 

\begin{lemma} \label{general set}
Let $X$ be a variety of dimension $d \geq 1$ and let $A,B,C \subseteq X$.
\begin{enumerate}
\item If $A$ is countably dense then $A$ is Zariski-dense. 
\item If  $A$ is very general then $A$ is countably dense (and hence Zariski-dense).
\item If $A$ is countably dense and $B$ is very general, then $A \cap B$ is countably dense.
\item If $A\setminus B \subseteq C$, with $A$  very general, then either $B$ is countably dense or $C$ contains a very general subset of $X$.
\end{enumerate}
\end{lemma}

If we have a family of points and divisors through them, then the countable density of the set of points is the right property that allows us to extract a finite number of divisors that are ``unrelated'', in a certain sense:

\begin{lemma} \label{tanti divisori}
Let $X$ be a variety of dimension $\geq 1$ and $A$ a countably dense subset of $X$. Suppose that for all $x \in A$ there exists a divisor $D_x$ such that $x \in \supp(D_x)$. Then there exist $x_1, x_2 \in A$ such that $x_1 \not \in \supp(D_{x_2})$ and $x_2  \not \in \supp(D_{x_1})$.

More generally, under the same hypotheses, for every $n \in \mathbb{N}$ there exist  $x_1, \ldots, x_n \in A$ such that $x_i \not \in \supp(D_{x_j})$ for every $i \not = j$, $1 \leq i,j \leq n$. 

\end{lemma}
\begin{proof}
Take a countable, Zariski-dense set $B \subset A$. For all $b \in B$ consider $D_b$. $V:=A \setminus \cup_{b \in B}\supp(D_b)$ is non-empty (otherwise $A \subseteq\cup_{b \in B}\supp(D_b)$, contradiction). Let $x_1 \in V$. Since $B$ is Zariski-dense, $D_{x_1}$ cannot pass through $b$ for every $b \in B$. Let $x_2$ such that $x_2 \not \in \supp(D_{x_1})$ and we are done. 

For the general case choose $B_1:=B$ as before. We define $B_2, \ldots, B_{n-1}$ inductively: suppose we have already defined $B_2, \ldots, B_i$; since $V_i:=A \setminus \cup_{k=1}^i  \cup_{b \in B_k}\supp(D_b)$ is still countably dense, and hence Zariski-dense, we can choose a countable Zariski-dense set $B_{i+1} \subset V_i$.  Now we define $x_1, \ldots, x_n$ inductively. Choose a point $x_n \in V_{n-1}$. Suppose we have already defined $x_n, x_{n-1}, \ldots, x_{i+1}$. Since $B_{i}$ is Zariski-dense, there exists a point $x_{i}$ such that it does not belong to $\cup_{k=i+1}^n\supp(D_{x_k})$. $x_1, \ldots x_n$, defined in this way, respect the requirements on the associated divisors, and we are done.
\end{proof}

Note that Zariski-density is not enough to obtain the same conclusion: for example consider, on a curve, a  countable infinity of points $\{x_1, \ldots, x_n, \ldots\}$ and, for every $x_n$, the divisor $D_n=x_1 + \ldots + x_n$.

When we will study pluricanonical systems on a projective variety $X$ it will be clear that we can have better explicit results if we know that we do not need to deal with curves of small volume (i.e.\ of small genus). That is why we give the following
\begin{definition}
Let $X$ be a projective variety. Let $g \in \mathbb{N}^+$. Let $$\Omega_g:= \bigcup_{ \begin{subarray}{r}  C  \text{ curve } \subseteq X,  \\  g(C) < g  \end{subarray}} C$$ (where $g(C)$ is the geometric genus of the (possibly singular) curve $C$). Then we will say that $X$ is $g$-countably dense if $\Omega_g$ is countably dense, that is: $\Omega_g$ is not contained in the union of countably many proper closed subvarieties of $X$. 
\end{definition}

\begin{remark} \label{notcountablydense}
Clearly, if $X$ is not $g$-countably dense then $X$ is not $g'$-countably dense for every $g' \leq g$. 
Moreover if $X$ is not $g$-countably dense then, by definition, there exists a very general subset $\Lambda$ of $X$ such that for every $x \in \Lambda$ and every curve $C$ through $x$ then $g(C) \geq g$.
\end{remark}

\begin{remark} \label{generaltype}
If $X$ is of general type then there exists a very general subset $\Lambda$ of $X$ such that every subvariety through any point of $\Lambda$ is of general type. Hence such an $X$ is not $2$-countably dense. 
\end{remark}

\subsection{Volume and big divisors} \label{volume}

\begin{definition} 
Let $X$ be a variety of dimension $d$ and let $D$ be a Cartier integral divisor. Then the \textit{volume of D}, $\vol(D)$, is just $\limsup_{m \rightarrow + \infty} \frac{h^0(X, mD)\cdot d!}{m^d}$. This limsup is actually a limit and the definition can be naturally extended to $\mathbb{Q}$-Cartier divisors. The volume of a divisor does depend only on its numerical class. If $X$ is nonsingular and $K_X$ is its canonical bundle then $\vol(X):=\vol(K_X)$. Since the volume of a divisor is a birational invariant then if $X$ is singular take any desingularization $X'$ of $X$ and set $\vol(X):=\vol(X')$. If $\vol(D) >0$ then $D$ is called \textit{big}. If $K_X$ is big then $X$ is called a variety of \textit{general type}. For all these matters see \cite[2.2.C]{LazI}.
\end{definition}

Thus the volume of an integral divisor measures the number of its sections, but only asymptotically. Even so, one can hope (in certain cases) to obtain information also about actual multiples of the divisor: the key point is to find a specific subvariety and then prove that the restriction map (for the given divisor) is surjective. Both to produce the subvariety and to study the surjectivity of the restriction map, one needs to use particular techniques, such as the Tie Breaking (see \cite[Proposition 8.7.1]{KollarlibroCorti} and \cite[Theorem 3.7]{Brustet}) or Nadel's vanishing theorem (see \cite[Theorem 9.4.8]{LazII}), that require the divisor to be ample (or big and nef). When the divisor is not ample but only big then we can use local analogues: in fact a big divisor is ample outside a closed subset. The following definitions and lemma will make this clearer: 

\begin{definition}
Let $X$ be a variety, let $D$ be a $\mathbb{Q}$-Cartier divisor and let $p \in \mathbb{N}^+$ be such that $pD$ is integral. The \textit{stable base locus} of  $D$ is the algebraic set $\mathbb{B}(D)= \bigcap_{m \geq 1} Bs(|mpD|)$ (cf. \cite[ \S 1]{ELMNP2} or \cite[Definition 2.1.20, Remark 2.1.24]{LazI}). Unfortunately these loci do not depend only on the numerical class of $D$. Nakamaye then suggested to slightly perturb $D$: the \textit{augmented base locus} of $D$ is defined as $\mathbb{B}_+(D)=\mathbb{B}(D-\epsilon A)$ for any ample $A$ and sufficiently small $\epsilon \in \mathbb Q$. This definition does not depend on $A$ or on $\epsilon$ (provided it is sufficiently small). Moreover $D$ is big if and only if $\mathbb{B}_+(D)$ is a proper closed subset of $X$ (see \cite[ \S 1, in particular Example 1.7]{ELMNP2}).
\end{definition}

\begin{lemma} \label{fuorisupporto}
Let $X$ be a projective variety and $D$ a big $\mathbb{Q}$-Cartier divisor on $X$. Take any norm $\| \cdot\|$ on $N^1(X)_{\mathbb{Q}}$. Then there exists $\epsilon >0$ such that for every ample $\mathbb{Q}$-Cartier divisor $A$, $\|A\|< \epsilon$, and for every $x \not \in \mathbb{B}_+(D)$ there is an effective $\mathbb{Q}$-Cartier divisor $E$ such that $x \not \in \supp(E)$ and $D \sim_{\mathbb{Q}} A+E$.
\end{lemma}

\begin{proof}
By \cite[\S 1]{ELMNP2}, \cite[10.3.2]{LazII} and \cite[2.1.21]{LazI}, there exists $m \in \mathbb{N}$ such that $mD$, $mA$ are integral divisors and $\mathbb{B}_+(D)=\mathbb{B}(D-A)=Bs(|mD-mA|)$. Since $x \not \in \mathbb{B}_+(D)$ then there exists an effective divisor $F \in |mD-mA|$ such that $x \not \in \supp(F)$. Set $E:= F/m$. $D \sim_{\mathbb{Q}} A+E$ and we are done. 
\end{proof}

\begin{remark} \label{fuorisupporto2}
We could have chosen $E$ to skip $n$ points not in $\mathbb{B}_+(D)$.
\end{remark}

\subsection{Multiplier ideals and singularities of pairs}

First of all we recall some standard definitions:

\begin{definition} (see \cite[9.1.10, 9.3.55]{LazII} and \cite[0.4]{Kollar}). A \textit{pair}  $(X,\Delta)$ consists of a normal variety $X$ and a $\mathbb{Q}$-divisor $\Delta$ such that $K_X+\Delta$ is a $\mathbb{Q}$-Cartier $\mathbb{Q}$-divisor. The pair $(X, \Delta)$ is said to be \textit{effective} if $\Delta$ is effective. A projective birational morphism $\mu: X' \rightarrow X$ is said to be a \textit{log resolution} of the pair $(X,\Delta)$ if $X'$ is smooth,  \text{Exc}$(\mu)$ is a divisor and $\mu^{-1}(\supp(\Delta)) \cup \text{Exc}(\mu)$ is a divisor with simple normal crossing support.
\end{definition}

\begin{definition} (see\ \cite[\S 9.2.A]{LazII}).
Let $X$ be a smooth variety and let $D$ be a $\mathbb{Q}$-divisor on $X$. The \textit{multiplier ideal sheaf} $\mathcal{J}(D)=\mathcal{J}(X,D)$ is defined in the following way: fix any log resolution $\mu: X' \rightarrow X$ of $(X,D)$; then $\mathcal{J}(D):= \mu_*\mathcal{O}_{X'}\left(K_{X'/X} - \left[ \mu^*D\right] \right)$.
\end{definition}

\begin{definition} 
Let $(X, \Delta)$ be a pair and $\mu: X' \rightarrow X$ be a log resolution of the pair. We can canonically write $K_{X'} - \mu^*(K_X + \Delta) \equiv \sum a(E) E$, where the sum is taken over all prime divisors $E$. Given $x \in X$, $(X, \Delta)$ is said to be \textit{klt at $x$} or \textit{kawamata log terminal at $x$} (respectively: \textit{lc at $x$} or \textit{log canonical at $x$}) if for every $E$ such that $x \in \mu(E)$ we have that $a(E) > -1$ (resp.: $a(E) \geq -1$). $(X, \Delta)$ is  \textit{klt} or \textit{kawamata log terminal} (respectively: \textit{lc} or \textit{log canonical}) if it is klt (resp.: lc) at $x$ for every $x \in X$. We say that a subvariety $V \subset X$ is a \textit{lc centre} or \textit{log canonical centre} for the pair $(X, \Delta)$ if it is the image, through a certain $\mu$, of a divisor $E$ such that  $a(E)\leq -1$. The valuation corresponding to this divisor is called a \textit{log canonical place}. A log canonical centre $V$ is \textit{pure} if it is log canonical at the generic point of $V$. A log canonical centre $V$ is \textit{exceptional} if it is pure and there is a unique log canonical place lying over the generic point of $V$. We will denote by $LLC(X, \Delta, x)$ the set of all lc centres that pass through $x$. 
\end{definition}

If $(X,D)$ is effective and $X$ is smooth, then we can use equivalent definitions for klt and lc: $(X, D)$ is klt if $\mathcal{J}(X,D)=\mathcal{O}_X$; $(X, D)$ is lc if $\mathcal{J}(X, (1-\epsilon)D)=\mathcal{O}_X$ for all $0 < \epsilon < 1$. Analogously for the local statements. This justifies the following:

\begin{definition}
Let $(X, D)$ be an effective pair, with $X$ smooth and let $x \in X$. The \textit{log canonical threshold at $x$}, $lct(D,x)=lct(X,D,x)$, is just $\inf\{c>0 | \mathcal{J}(X,cD)_x \subsetneq \mathcal{O}_{X,x}\}$.  We will denote by $Nklt(X,D)$ the \textit{non-klt locus for $(X,D)$}, i.e.\ $\supp(\mathcal{O}_X/\mathcal{J}(X,D)) \subset X$ with the reduced structure. 
\end{definition}

Log canonical centres will be our main tool to produce subvarieties from which it is possible to pull back forms. Log canonical centres, in our case, are quite well behaved from this point of view: they can be made exceptional (using the Tie Breaking method) and their dimension can be cut down (see the original work by Angehrn--Siu in \cite{AngehrnSiu} and also \cite[Theorem 4.1]{HcMcK}, \cite[ \S 5]{Takayama}, \cite[Proposition 10.4.10]{LazII}). 

The following lemma about log canonical centres of codimension $1$ will be needed later:
\begin{lemma} \label{centrocod1}
Let $(X, \Delta)$ be a pair,  $\Delta=\sum_{i=1}^sd_i \Delta_i$ with $\Delta_i$ prime divisors and $d_i \in \mathbb{Q}$. If $W$ is a lc centre for $(X,\Delta)$ of codimension $1$ then there exists $\overline{i} \in \{1,\ldots,s\}$ such that $W=\Delta_{\overline{i}}$ and $d_{\overline{i}} \geq 1$. If moreover $W$ is pure then $d_{\overline{i}}=1$.
\end{lemma}

\begin{proof}
By definition of lc centre, there exists  $\mu: X' \rightarrow X$ a log resolution of $(X, \Delta)$ and a prime divisor $E \subset X'$ such that $\mu(E)=W$ and of discrepancy $a(E,X,\Delta) \leq -1$. Since the codimension of $W$ is $1$ then $E$ cannot be exceptional for $\mu$, hence (cf.\ \cite[9.3G, footnote 14]{LazII} or \cite[2.25-2.26]{Kollar}) $E$ is a strict transform of one of the $\Delta_i's$, i.e $\exists \overline{i}$ such that $W=\mu(E)=\Delta_{\overline{i}}$ and $a(E,X,\Delta)=-d_{\overline{i}} \Rightarrow d_{\overline{i}} \geq 1$.

If moreover $W$ is pure, i.e. it is lc at the generic point of $W$, then, since $\mu(E)=W$ actually contains the generic point of $W$, $-d_{\overline{i}}=a(E,X,\Delta)\geq -1 \Rightarrow d_{\overline{i}}=1$.
\end{proof}

\subsection{Some techniques}
In this section we list some of the techniques that will be involved later. Most of them are already well-known but since they are needed in more particular settings sometimes we include also proofs.

First of all we state the classical Tie Breaking theorem, but in its local version, using big divisors to perturb the log canonical centre, instead of ample ones. Check also \cite[Proposition 8.7.1]{KollarlibroCorti}  and \cite[Theorem 3.7]{Brustet}.

\begin{lemma} [local Tie Breaking with a big divisor](cf.\ \cite[Lemma 2.6]{Todorov}) \label{tiebreaking}
Let $X$ be a complex smooth projective variety and $\Delta$ an effective $\mathbb{Q}$-divisor and assume that $(X, \Delta)$ is lc but not klt at some point $x \in X$. Then:
\begin{enumerate}[a.]
\item \label{zero} If $W_1, W_2 \in LLC(X, \Delta, x)$ and $W$ is an irreducible component of $W_1 \cap W_2$ containing $x$, then $W \in LLC(X,\Delta, x)$.
\item By the item before, $LLC(X, \Delta, x)$ has a unique minimal irreducible element, say $V$.
\item \label{uno} If $L$ is a big divisor and $x \not \in \mathbb{B}_+(L)$ then there exist a positive rational number $a$ and an effective $\mathbb{Q}$-divisor $M$ such that $M \sim_{\mathbb{Q}} L$ and such that for all $ 0 < \epsilon \ll 1$, $(X, (1-\epsilon)\Delta+\epsilon aM)$ is lc at $x$ and $LLC(X,(1-\epsilon)\Delta+\epsilon aM, x)=\{V\}$. %(i.e.: $V$ is a pure lc centre).
\item \label{due} If $\Delta$ is big, $x \not \in \mathbb{B}_+(\Delta)$ and $\Delta \sim \lambda D$ with $\lambda < c$, $\lambda \in \mathbb{Q}^+$ and $D$ a $\mathbb{Q}$-divisor, then there exists an effective $\mathbb{Q}$-divisor $\Delta'$ such that $(X, \Delta')$ is lc, not klt at $x$, $LLC(X, \Delta',x)=\{V\}$ and $\Delta' \sim \lambda'D$ with $\lambda'<c$, $\lambda' \in \mathbb{Q}^+$.
\item \label{tre} In every case, we can also assume that there is a unique place lying above $V$, locally at $x$.
\end{enumerate}
\end{lemma}

The next lemma, due to Hacon--McKernan (cf. \cite[Lemma 2.6]{HcMcK}), essentially explains how to pull back sections from log canonical centres when we already know that these centres have dimension $0$. The main ingredient is Nadel's Vanishing theorem, that, under particular conditions, assures the surjectivity of the restriction map. When dealing with more that one point, this lemma will be applied together with Lemma \ref{tanti divisori}.

\begin{lemma}  \label{nonvanishing}
Let $X$ be a smooth projective variety and $D$ a big and integral divisor on $X$. Let $x,y \not \in\mathbb{B}_+(D)$. 
Assume that there exists an effective $\mathbb{Q}$-divisor $\Delta_x \sim_{\mathbb{Q}} \lambda_x D$ with $\lambda_x \in \mathbb{Q}^+$ and such that $LLC(X,\Delta_x,x)=\{\{x\}\}$. Then for every $m \in \mathbb{N}^+$ such that $m > [\lambda_x]$, $h^0(\Ox(K_X+mD))>0$.  If moreover there exists another effective $\mathbb{Q}$-divisor $\Delta_y \sim_{\mathbb{Q}} \lambda_y D$ with $\lambda_y \in \mathbb{Q}^+$, such that $LLC(X,\Delta_y,y)=\{\{y\}\}$ and such that $x \not \in \supp( \Delta_y)$ and $y \not \in \supp( \Delta_x)$, then for every $m \in \mathbb{N}^+$ such that $m > [\lambda_x+\lambda_y]$, $h^0(\Ox(K_X+mD))\geq2$.

More generally, let $x_1, \ldots, x_n \not \in \mathbb{B}_+(D)$. If for every $ 1 \leq i \leq n$ there exists an effective $\mathbb{Q}$-divisor $\Delta_i \sim_{\mathbb{Q}} \lambda_i D$ with $\lambda_i \in \mathbb{Q}^+$, such that $LLC(X, \Delta_i, x_i)=\{x_i\}$ and such that $x_i \not \in \cup_{j\not = i} \supp(\Delta_j)$ then for every $m \in \mathbb{N}^+$ such that $m > \left[ \sum_{i=1}^n \lambda_i \right]$, $h^0(\Ox(K_X+mD)) \geq n$.
\end{lemma}

\begin{proof}
Since $x \not \in \mathbb{B}_+(D)$, by Lemma \ref{fuorisupporto} there exist an ample $\mathbb{Q}$-divisor $A_x$ of sufficiently small norm and an effective $\mathbb{Q}$-divisor $E_x$ such that $D \sim_{\mathbb{Q}} A_x+E_x$ and $x \not \in \supp(E_x)$. Let us consider the multiplier ideal associated to $\Delta_x$, $\mathcal{J}(\Delta_x)$. Let us notice that, by the hypothesis that $\{x\}$ is an isolated lc-centre at $x$, there exists an open neighbourhood $U_x$ of $x$  such that $\mathcal{J}(\Delta_x)_x \subsetneq \mathcal{O}_{X,x}$ but $\mathcal{J}(\Delta_x)_z =\mathcal{O}_{X,z}$ for all $z \in U_x-\{x\}$.

Let $B_x$ be the $\mathbb{Q}$-divisor $\Delta_x + (m-\lambda_x)E_x$. Since $x \not \in \supp(E_x)$ we can say that $\mathcal{J}(B_x)_x \subsetneq \mathcal{O}_{X,x}$ and $\mathcal{J}(B_x)_z = \mathcal{O}_{X,z}$ for every $z \in U'_x:=U_x \cap (X-\supp(E_x))$, that is: the set of zeroes $Z(\mathcal{J}(B_x))$ has $x$ as an isolated point. 

Let us consider the following exact sequence:
$$ 0 \rightarrow \mathcal{J}(B_x) \rightarrow \mathcal{O}_X \rightarrow \mathcal{O}_{Z(\mathcal{J}(B_x))} \rightarrow 0$$
Tensoring it by $\Ox({K_X+mD})$ we obtain:
$$ 0 \rightarrow \mathcal{J}(B_x) \otimes \Ox(K_X+mD) \rightarrow \mathcal{O}_X(K_X+mD) \rightarrow \mathcal{O}_{Z(\mathcal{J}(B_x))} \otimes \Ox(K_X+mD) \rightarrow 0$$
$x$ is an isolated point in $Z(\mathcal{J}(B_x))$ so we have $h^0(\mathcal{O}_{Z(\mathcal{J}(B_x))}\otimes \Ox(K_X+mD)) > 0$. 

Let us notice that since $m$ is an integer greater than $[\lambda_x]$ then $m > \lambda_x$, hence  $mD-B_x \sim_{\mathbb{Q}} (m-\lambda_x) A_x$ is big and nef. Therefore we can apply Nadel's theorem (cf.\ \cite[Theorem 9.4.8]{LazII}) to conclude that $H^1(\mathcal{O}_X(K_X+mD) \otimes \mathcal{J}(B_x))=0$ and thus the first part of the lemma is proved.

Since $x,y \not \in \mathbb{B}_+(D)$ then, by Remark \ref{fuorisupporto2}, there exist an ample $\mathbb{Q}$-divisor $A$ of sufficiently small norm and an effective $\mathbb{Q}$-divisor $E$ such that $D \sim_{\mathbb{Q}} A+E$ and $x,y \not \in \supp(E)$.
Let $B$ be the $\mathbb{Q}$-divisor $\Delta_x+(m-\lambda_x-\lambda_y)E+\Delta_y$. Since $x,y \not \in \supp(E)$, $x \not \in \supp(\Delta_y)$, $y \not \in \supp(\Delta_x)$, we can conclude that $Z(\mathcal{J}(B))$ has $x,y$ as two isolated points.

Let us consider the following exact sequence:
$$ 0 \rightarrow \mathcal{J}(B) \otimes \Ox(K_X+mD) \rightarrow \mathcal{O}_X(K_X+mD) \rightarrow \mathcal{O}_{Z(\mathcal{J}(B))} \otimes \Ox(K_X+mD) \rightarrow 0$$
$x,y$ are isolated points in $Z(\mathcal{J}(B))$ so we have $h^0(\mathcal{O}_{Z(\mathcal{J}(B))}\otimes \Ox(K_X+mD)) \geq 2$ and since $mD-B \sim_{\mathbb{Q}} (m-\lambda_x-\lambda_y)A$ we have that $mD-B$ is big and nef (by hypothesis $\lambda_x+\lambda_y<m$). Therefore we can conclude as before, simply applying Nadel's theorem.  

The general case in analogous.
\end{proof}

%\begin{theorem} [Todorov] (cf. \cite{Todorov}, thm. 1.1)
%Let $X$ be a projective threefold of general type and assume that $vol(X)>2304^3$. Then $P_2=h^0(X, \Ox(2K_X))>0$.
%\end{theorem}
%\begin{proof}
%We will review here, without computations, the first part of the proof in \cite{Todorov}.
%Let $\mathcal{A}=\{Z \subseteq X, \dim Z >0, Z \textrm{ not of general type}\} $. Let $T=\bigcup_{Z \in \mathcal{A}} Z$. $T$ is a countable union of maximal (in terms of inclusions) subvarieties not of general type. Let $X_0$ be the complement of $T \cup \mathbb{B}_+(K_X)$. By construction $X_0$ the complement of a countable union of subvarieties. Let $x$ be a very general point, more precisely let $x \in X_0$. 
%Recall that we are supposing $\vol(X) > \alpha^3$.
%After some preliminary statements and a small perturbation we are in the following situation:  $D$ is a big $\mathbb{Q}$-divisor on $X$ with $D \sim \lambda K_X$ ($0<\lambda<3/\alpha$, $\lambda \in \mathbb{Q}$) and such that $(X,D)$ is lc at $x$ and $LLC(X,D,x)=\{V\}$, where $V$ is a lc centre, hence irreducible.
%Let us analyze different cases based on the dimension of this log canonical centre.
%\begin{enumerate}
%\item {$\dim(V)=0$}.
 %Making $\lambda<1$ (i.e.: $ \alpha \geq 3$) we simply apply \ref{nonvanishing} and we are done.
%\item $\dim(V)=1$.  Let $f: W \rightarrow V$ be a resolution of $V$. 
%\end{enumerate}
%\end{proof}

As we have already said, we will use log canonical centres to pull back sections of multiples of the canonical divisor. Unluckily this is not easy to do, unless the lc centres are points. Unfortunately when the volume is low, cutting down the dimension of lc centres does not allow us to have information about small multiple of the canonical divisor. That is why Todorov in \cite{Todorov}, using ideas of McKernan (see \cite{McK}) has developed another strategy in the case of threefold, that is to produce a morphism from the threefold to a curve. The next proposition shows how to create sections in this way:

\begin{proposition} \label{propositionfibration}
Let $X$ be a smooth projective threefold of general type. Suppose that there exist a smooth projective curve $B$ and a dominant morphism with connected fibres $f: X \rightarrow B$ such that the general fibre $X_b$ is a minimal, smooth surface of general type. Moreover suppose there exist $\lambda \in \mathbb{Q}^+$ and, for a general $b \in B$, an effective $\mathbb{Q}$-divisor $D_b$ on $X$ such that $D_b \sim_{\mathbb{Q}} \lambda K_X$ and such that $X_b$ is a lc centre for $(X, D_b)$. Suppose also that, for general $b$, there exists $\beta \in \mathbb{Q}^+$ such that $\vol(X_b) \leq \beta^2$. Then, given $b_1, \ldots, b_k$ general points on $B$,  the restriction map gives a surjection
$$H^0(\mathcal{O}_X((n+1)K_X) \rightarrow H^0(\mathcal{O}_{X_{b_1}}((n+1)K_{X_{b_1}})) \oplus \ldots \oplus H^0(\mathcal{O}_{X_{b_k}}((n+1)K_{X_{b_k}}))$$
as long as $ \lambda k (4(n+1)[\beta^2]- 1) <  1$ and $n > \lambda k$.
\end{proposition}

\begin{proof}
By Kawamata's theorem A (cf. \cite{Kawamata98}, taking $S=\{pt\}$) for every $1 \leq i \leq k$ and every positive integer $m$ the restriction maps $H^0(\mathcal{O}_X(m(K_X+X_{b_i}))) \rightarrow H^0(\mathcal{O}_{X_{b_i}}(mK_{X_{b_i}}))$ are surjective. Since for every $i$ we have an injection $H^0(\mathcal{O}_X(m(K_X+X_{b_i}))) \hookrightarrow H^0(\mathcal{O}_X(m(K_X+X_{b_1}+\ldots+{X_{b_k}})))$, then the restriction maps $H^0(\mathcal{O}_X(m(K_X+X_{b_1}+\ldots+X_{b_k}))) \rightarrow H^0(\mathcal{O}_{X_{b_i}}(mK_{X_{b_i}}))$ are surjective. Since $X_{b_i}$ are minimal surfaces of general type then, by \cite{Bombieri} for $m$ large enough (namely $m \geq 4$) $|mK_{X_{b_i}}| $ is base point free, hence  a general $G \in |m(K_X+X_{b_1}+\ldots+X_{b_k})|$ is such that for every $i$, $G|_{X_{b_i}}$ is a general divisor in the base-point-free linear system $|mK_{X_{b_i}}|$.

Since $K_X$ is big then $K_X=A+E$ where $A$ is an ample $\mathbb{Q}$-divisor and $E$ an effective $\mathbb{Q}$-divisor.

Let now $b'$ be a general point on $B$, $m$ be a sufficiently large integer, $G$ a general divisor in $ |m(K_X+X_{b_1}+\ldots+X_{b_k})|$, $\epsilon$ a rational number, $0<\epsilon \ll 1$. 
Let $h:=:h_{n,k}:=\frac{k(n+1)-\epsilon k}{\lambda k +1},$ $j:=:j_{n,k}:=-h_{n,k},$ $i:=:i_{n,k}:=-1+\frac{h_{n,k}}{k},$ and consider the $\mathbb{Q}$-divisor $F:=:F_{n,k}:= hD_{b'}+ \frac{i}{m}G+jX_{b'}+\epsilon E$. For $\epsilon$ sufficiently small $h > 0$. Since $X_{b'}$ is an exceptional log canonical centre of $D_{b'}$ then $D_{b'}=X_{b'}+$ other surfaces. Therefore if $i \geq 0$ then $F$ is an effective divisor: in order to have $i \geq 0$ it is enough to ask that $n > \lambda k$.

Moreover, by the choices of $h,i,j$, we have that $nK_X-(X_{b_1}+\ldots+X_{b_k})-F \equiv \epsilon A$.

Start with the following short exact sequence:
$$ 0 \rightarrow \mathcal{O}_X(-(X_{b_1}+\ldots +X_{b_k})) \rightarrow \mathcal{O}_X \rightarrow\mathcal{O}_{X_{b_1}} \oplus \ldots \oplus \mathcal{O}_{X_{b_k}} \rightarrow 0$$

After tensoring it by $\mathcal{O}_X((n+1)K_X)$ we have the following exact sequence (see \cite[Remark 1.39]{DiBiagio}): $$0 \rightarrow \mathcal{O}_X((n+1)K_X-(X_{b_1}+\ldots+X_{b_k})) \otimes \mathcal{J}(F) \rightarrow \mathcal{O}_X((n+1)K_X) \otimes \mathcal{J}(F) \rightarrow$$  $$\rightarrow \mathcal{O}_{X_{b_1}}((n+1)K_{X_{b_1}}) \otimes \mathcal{J}(F)_{X_{b_1}} \oplus \ldots \oplus\mathcal{O}_{X_{b_k}}((n+1)K_{X_{b_k}}) \otimes \mathcal{J}(F)_{X_{b_k}} \rightarrow 0$$

By Nadel's vanishing theorem (see \cite[9.4.8]{LazII}), $H^1( \mathcal{O}_X((n+1)K_X-(X_{b_1}+\ldots+X_{b_k})) \otimes \mathcal{J}(F) )=0$. Moreover, since $F$ is effective, $\mathcal{J}(F) \subseteq \mathcal{O}_X$, hence $\mathcal{O}_{X}((n+1)K_X)\otimes \mathcal{J}(F) \subseteq \mathcal{O}_X((n+1)K_X)$. Therefore to prove the theorem it is now sufficient only to prove that, under the hypotheses, $\mathcal{J}(F)_{X_{b_i}}$ is trivial for every $i$.

To ease the notation, let $b=b_i$. Since $X_b \nsubseteq \supp(F)$, we have that $\mathcal{J}(F)_{X_{b}} \supseteq \mathcal{J}(F|_{X_b})$, therefore we have to prove only that $\mathcal{J}(F|_{X_b})$ is trivial. Set $\Delta:=D_{b'}|_{X_b}$ and $\Gamma:=E|_{X_b}$. $\Delta$ and $\Gamma$ are effective divisors, with $\Delta \sim_{\mathbb{Q}} \lambda K_{X_b}$. $F|_{X_b}=h\Delta+ \frac{i}{m} G|_{X_b} + \epsilon \Gamma$.  Since $m$ is large enough and $G|_{X_b}$ is a general divisor in the base-point-free linear system $|mK_{X_{b}}|$, then, by Kollar--Bertini (see \cite[9.2.29]{LazII}), $\mathcal{J}(h\Delta+ \frac{i}{m} G|_{X_b} + \epsilon \Gamma)=\mathcal{J}(h\Delta+  \epsilon \Gamma)$. But, by \cite[Proposition 9.2.32.i]{LazII}, $\mathcal{J}(h\Delta+  \epsilon \Gamma) \supseteq \mathcal{J}(h\Delta+ \frac{\epsilon k}{\lambda k +1 } \Delta+ \epsilon \Gamma)=\mathcal{J}(\frac{k(n+1)}{\lambda k +1} \Delta + \epsilon \Gamma)=\mathcal{J}(\frac{k(n+1)}{\lambda k +1} \Delta )$, where the last equality is due to \cite[Example 9.2.30]{LazII}. Set $h':=:h'_{n,k}=\frac{k(n+1)}{\lambda k +1}$. Now for every $x \in \Delta$, pick a curve $C\subset X_b$ passing through $x$ that is a component of a divisor in $|4K_{X_b}|$ but it is not a component of $\Delta$ (cf.\ \cite[proof of claim 1]{Todorov}). %This can be done by \ref{superficie}, choosing a general hyperplane $H$ and noticing that $h^0(4K_{X_b}) \geq 4$ and that $\phi_{|4K_{X_b}|}(X_b)$ does not contain lines. 
Then $\mult_x(h'\Delta)= h' \mult_x(\Delta) \leq h'\Delta.C$. Since $\Delta \sim_{\mathbb{Q}} \lambda K_{X_b}$ is nef ($X_b$ is minimal and of general type) then $h' \Delta.C \leq 4 h' \Delta.K_{X_b} =  4 h' \lambda K_{X_b}^2$.  If $\mult_x(h' \Delta) <1$ for all $x \in \Delta$ then $\mathcal{J}(h' \Delta)$ is trivial, as wanted (cf.\ \cite[Proposition 9.5.13]{LazII}). Therefore we need only to impose $4 \lambda h' K_{X_b}^2 < 1$. By hypothesis, and since $\vol(X_b)$ is an integer, it is enough to ask that $\lambda k (4(n+1) [\beta^2] -1) < 1$.
\end{proof}

\section{Plurigenera for 3-folds of general type}

\begin{theorem} \label{threefold}
Let $X$ be a smooth projective threefold of general type such that $\vol(X) > \alpha^3$.  If  $\alpha \geq 879$ then $h^0(2K_X) \geq 1$ and if  $\alpha \geq 432(n+1)-3$ then $h^0((n+1)K_X)\geq n$, for all $n \geq 2$. More generally, if $X$ is not $g$-countably dense and if $g,n,\alpha$ are as in Table \ref{tavolanonvanishing} or, in the other cases, $\alpha \geq 48(n+1)-3$, then $h^0((n+1)K_X)\geq n$, for all $n \geq 1$. Moreover, under the same bounds on $\alpha$ and $g$ given by the case $n=1$, we have that $h^0(lK_X) \geq 1$ for all $l \geq 2$.
\end{theorem}

\begin{table}[!ht] \caption{} \label{tavolanonvanishing}
\begin{center}
\begin{tabular}{c|c|c||c|c|c}
$g$ &$n$ & $\alpha$ & $g$ &$n$ & $\alpha$ \\ \cline{1-6}
$2$ & $1$ & $\geq 879$ & $ 10$ & $1, \ldots, 304 $ & $\geq 60(n+1)-3 $ \\  \cline{2-3} \cline{5-6}
 & $\geq 2$ & $\geq 432(n+1)-3$ & $ $ & $305, \ldots, 381 $ & $\geq 18354$ \\ \cline{1-3}  \cline{4-6}
$3$ & $\geq 1 $ & $\geq 132(n+1)-3 $ & $ 11$ & $1, \ldots, 8$ & $\geq 60(n+1)-3$ \\ \cline{1-3} \cline{5-6}
$ 4$ & $1, \ldots, 6 $ & $\geq 96(n+1)-3 $ & $ $ & $9,10 $ & $\geq 550$ \\ \cline{4-6} \cline{2-3}
$ $ & $7$ & $\geq 714 $&$ 12$ & $1, \ldots, 4$ & $\geq 60(n+1)-3$ \\ \cline{5-6}  \cline{2-3}
$ $ & $ \geq 8$ & $\geq84(n+1)-3 $ & $ $ & $5 $ & $\geq 306$ \\ \cline{4-6} \cline{1-3}
$5$ & $1 $ & $\geq 165 $ & $ 13$ & $1,2$ & $\geq 60(n+1)-3$ \\ \cline{5-6} \cline{2-3}
$ $ & $2 $ & $\geq 242$ & $ $ & $3 $ & $\geq 223$ \\ \cline{4-6} \cline{2-3}
$ $ & $\geq 3 $ & $\geq 72(n+1)-3$& $ 14$ & $1,2$ & $\geq 60(n+1)-3$ \\ \cline{4-6} \cline{1-3}
$ 6$ & $1, \ldots, 43 $ & $\geq 72(n+1)-3 $&$ 15$ & $1$ & $\geq 117$ \\ \cline{5-6} \cline{2-3}
$ $ & $44, \ldots, 52 $ & $\geq 3234$ &$ $ & $2$ & $\geq 156$ \\ \cline{4-6} \cline{2-3}
$ $ & $\geq 53$ & $\geq 60(n+1)-3$ &$16,17,18 $ & $1$ & $\geq 117$ \\ \cline{4-6} \cline{1-3}
$7 $ & $1$ & $\geq 141$ &$19 $ & $1$ & $\geq 111$ \\ \cline{4-6} \cline{2-3}
$ $ & $2 $ & $\geq 184$ & $20 $ & $1$ & $\geq 105$ \\ \cline{4-6} \cline{2-3}
$ $ & $\geq 3$ & $\geq 60(n+1)-3$&$21 $ & $1$ & $\geq 101$ \\ \cline{4-6} \cline{1-3}
$ 8,9$ & $\geq  1$ & $\geq 60(n+1)-3$&$22 $ & $1$ & $\geq 97$ \\ \cline{4-6} \cline{1-3}
\end{tabular}
\end{center}
\end{table}

%\begin{remark}
%Few cases (depending on $n$, $g$) have different numerical conditions on $\alpha$ - see the list at the end of the proof. 
%\end{remark}

\begin{remark}
This improves \cite[Theorem 1.1]{Todorov}.
\end{remark}

\begin{proof}
We will follow \cite{Todorov} very closely. Since we need to obtain explicit numbers from an asymptotic measure (the volume) the idea is to use the hypothesis about the volume to produce singular divisors and, in this way, log canonical centres. Then we would like to pull back sections from the log canonical centres, using Nadel's vanishing theorem. Unfortunately we do not have information about sections of  systems of divisors on lc centres, unless lc centres are points: thus we need a technique by Hacon--McKernan to cut down the dimension of the lc centres. But when lc centres have codimension $1$ and small volume this cutting-down process does not lead to have a bicanonical section: therefore Todorov's idea is to apply, in this case, a theorem of McKernan about family of tigers and so produce a fibration of $X$ onto a curve and then, from this fibration, produce bicanonical sections (cf.\ Proposition \ref{propositionfibration}).

%Without loss of generality we can replace $X$ by a desingularization. 
By Remark \ref{generaltype} $X$ is at least not $2$-countably dense, hence $g\geq 2$. Furthermore, since $X$ is not $g$-countably dense then by Remark \ref{notcountablydense} there exists a very general subset $\Lambda$ such that every curve passing through any point of $\Lambda$ has geometric genus $ \geq g$. Let $X_0$ be the intersection between $\Lambda$ and the complement of the union of all subvarieties of $X$ not of general type and $\mathbb{B}_+(K_X)$. $X_0$ is a very general subset of $X$, hence countably dense. 

Since $\vol(K_X) > \alpha^3$, by \cite[2.2]{Todorov}, for every $x \in X$ and every $k \gg 0$ there exists a divisor $A_x \in |kK_X|$ with $\mult_x(A_x) > k \alpha$. Let $\Delta'_x:= A_x \frac{\lambda'_x}{k}$, with $\lambda'_x < \frac{3}{\alpha}$, $\lambda'_x \in \mathbb{Q}^+$, but close enough to $\frac{3}{\alpha}$ so that $\mult_x(\Delta'_x) > 3$. Note that $\Delta'_x \sim \lambda_x'K_X$. Let $s_x:=lct(X, \Delta'_x, x)$. By \cite[9.3.2 and 9.3.12]{LazII}, $s_x < 1$ . Moreover, by \cite[9.3.16]{LazII}, $s_x \in \mathbb{Q}^+$. Therefore, without loss of generality, we can suppose that $(X, \Delta'_x)$ is lc, not klt in $x$.

By Lemma \ref{tiebreaking}, \ref{due}., for every $x \in X_0$ there exists an effective $\mathbb{Q}$-divisor $D_x \sim \lambda_x K_X$, with $\lambda_x < \frac{3}{\alpha}, \lambda_x \in \mathbb{Q}^+$, such that $(X, D_x)$ is lc, not klt in $x$ and $LLC(X, D_x, x)=\{V_x\}$, where $V_x$ is the unique minimal irreducible element of $LLC(X, \Delta'_x, x)$. Moreover we can also assume that $V_x$ is an exceptional lc centre.

Fix $\beta \in \mathbb{Q}^+$. Set $$Y_0:=\{x \in X_0 \textrm{ s.t. } \dim(V_x)=0\},$$ $$Y_1:=\{x \in X_0 \textrm{ s.t. } \dim(V_x)=1\},$$ $$Y_{2,a}:=\{x \in X_0 \textrm{ s.t. } \dim(V_x)=2 \textrm{ and } \vol(K_{V_x})> \beta^2\},$$ $$Y_{2,b}:=\{x \in X_0 \textrm{ s.t. } \dim(V_x)=2 \textrm{ and } \vol(K_{V_x}) \leq \beta^2\}.$$

Since $X_0$ is countably dense then at least one between $Y_0$, $Y_1$, $Y_{2,a}$ and $Y_{2,b}$ is countably dense. We will therefore analyze these cases.

%Before continuing with the proof we remark that since given  a variety of general type (i.e. $\vol > 0$), all the subvarieties that pass through a very general point are of general type (i.e. $\vol>0$), one is tempted to argue that if the volume is sufficiently large then all the subvarieties that pass through a general point have large volume. Unfortunately this is not the case: for example let  $C_g$ be a smooth curve of genus $g$. Then $\vol(C_2 \times C_g)=\vol(C_2)\vol(C_g)=4(2g-2)$. Hence we can construct examples of varieties of volume as large as one likes but such that for every point there is a curve of $\vol=2$ (namely $C_2$).

First of all, let us assume that $Y_0$ is countably dense.
For every $x \in Y_0$ we have that $V_x=\{x\}$, in fact $\dim(V_x)=0$ and $V_x$ is irreducible. $x \in \supp(D_x)$. Therefore we can apply Lemma \ref{tanti divisori} and Lemma \ref{nonvanishing} (with $m=n$) to conclude that for every $n \geq 1$, as soon as
\begin{equation} \label{a-1}
\lambda_x < \frac{3}{\alpha} \leq 1 \Leftrightarrow \alpha \geq  3,
\end{equation}
$h^0(K_X + nK_X)=h^0((n+1)K_X) \geq n$.

Let us now consider the case $Y_1$ countably dense. We wish to apply \cite[Theorem 4.1]{HcMcK} to cut down the dimension of the lc centers. For every $x \in Y_1$ consider $V_x$ and a resolution $f_x: W_x \rightarrow V_x$. As we have already seen, $V_x$ is an exceptional lc centre of $(X, D_x)$. Since $x \in X_0$, $V_x$, and hence $W_x$, are of general type and $V_x$ is not contained in the augmented base locus of $K_X$. Moreover $g(W_x) \geq g \geq 2$. Let $U_x $ be the very general subset of $V_x$ defined as in \cite[Theorem 4.1]{HcMcK}.  Set $U'_x:=U_x \cap X_0$. $U'_x$ is still a very general and non-empty subset of $V_x$. We also have that $\vol({W_x}) \geq 2g-2$. Let $\epsilon \ll 1$, $\epsilon \in \mathbb{Q}^+$. Then $\vol(W_x) > 2g-2 - \epsilon$.  Set $t_1:= 1/(2g-2-\epsilon)$: $\vol(t_1 K_{W_x}) > 1$. 
For every $y \in U'_x$ let us consider $y' \in f_x^{-1}(y) \subset W_x$. Since $y'$ is a smooth point, by \cite[2.2]{Todorov} and \cite[9.3.2]{LazII}, there exists $\Theta_{y'} \sim t_1 K_{W_x}$ such that $(W_x, \Theta_{y'})$ is not klt in $y'$. As before, since $lct(W_x, \Theta_{y'}, y') < 1$, we can suppose that $(W_x, \Theta_{y'})$ is lc, not klt in $y'$ and $\Theta_{y'} \sim \mu_{y'} K_{W_x}$ with $\mu_{y'} \in \mathbb{Q}^+$ and $\mu_{y'} \leq 1/(2g-2-\epsilon)$. Since $W_x$ is a curve and lc centres are irreducible, $LLC(W_x, \Theta_{y'},y')=\{y'\}$. We can now apply \cite[Theorem 4.1]{HcMcK}, since $f_x(y')= y \in U'_x$ and $\{y'\}$ is a pure lc centre: for every $\delta \in \mathbb{Q}^+$, there exists a divisor $D'_{y}$ such that $\{y\}$ is an exceptional lc centre for $(X, D'_{y})$ and $D'_{y} \sim ((\lambda_x +1)(\mu_{y'}+1)-1+\delta)K_X$. Let us notice that since $\{y\}$ is an exceptional lc centre then $LLC(X, D'_y, y)=\{y\}$. 

At the end we have the following situation: for every point $z \in \cup_{x \in Y_1} U'_x$ there exists a $\mathbb{Q}$-divisor $D'_z$ such that $LLC(X, D'_z, z)=\{z\}$ and such that $D'_z \sim ((\lambda_z+1)(\mu_z+1)-1 + \delta)K_X$, with $\lambda_z < \frac{3}{\alpha}$ and $\mu_z \leq 1/(2g-2-\epsilon)$. Let us prove that  $\cup_{x \in Y_1} U'_x$ is still a countably dense subset of $X$: if $\cup_{x \in Y_1} U'_x \subseteq \cup_{i \in \mathbb{N}} Z_i$, where $Z_i$ are closed proper subset of $X$, then, for every $x \in Y_1$, $U'_x \subseteq \cup_{i \in \mathbb{N}} Z_i$, hence $U'_x \subseteq (\cup _{i \in \mathbb{N}} Z_i) \cap V_x = \cup_{i \in \mathbb{N}} (Z_i \cap V_x)$. But $U'_x$ is very general in $V_x$, hence countably dense. Therefore for every $x \in Y_1$ there exists $i \in \mathbb{N}$ such that $Z_i \supseteq V_x \ni x$, i.e. $Y_1 \subseteq \cup_{i \in \mathbb{N}} Z_i$, but this is a contradiction. 

We can now apply Lemma \ref{tanti divisori} and Lemma \ref{nonvanishing} (with $m=n$) to conclude that for every $n \geq 1$ if 
\begin{equation} \label{a0}
\left(\frac{3}{\alpha}+1\right)(1+1/(2g-2-\epsilon))-1+ \delta \leq 1 \Leftrightarrow \alpha \geq \frac{6g-3-3\epsilon}{(2g-2-\epsilon)(1 - \delta)-1}
\end{equation}
(we are considering $\epsilon, \delta$ very small) then $h^0((n+1)K_X) \geq n$.

Let us now suppose that $Y_{2,a}$ is countably dense. Again, we want to apply \cite[Theorem 4.1]{HcMcK}.  As before, for every $x \in Y_{2,a}$  we have $V_x$, a resolution $f_x: W_x \rightarrow V_x$ and  $U_x $ the very general subset of $V_x$ defined as in \cite[Theorem 4.1]{HcMcK}. As before, consider $U'_x := U_x \cap X_0$. $U'_x$ is still a very general and non-empty subset of $V_x$. For every $y \in U'_x$ consider $y' \in f^{-1}_x(y)$. Since $\vol(V_x) > \beta^2$ then $\vol(W_x) = \vol(V_x) > \beta^2$. Set $t_1=2/\beta$. Then $\vol(t_1K_{W_x}) > 2^2$. Hence there exists $\Theta_{y'} \sim t_1K_{W_x}$ such that $(W_x, \Theta_{y'})$ is not klt in $y'$. Since $lct(W_x, \Theta_{y'}, y')<1$, we can suppose that $(W_x, \Theta_{y'})$ is lc, not klt in $y'$ and $\Theta_{y'} \sim \mu_{y'} K_{W_x}$ with $\mu_{y'} \in \mathbb{Q}^+$ and $\mu_{y'} < 2/\beta$. Therefore there exists a pure lc centre $W'_{y'} \in LLC(W_x, \Theta_{y'}, y')$. Set $V'_y:= f_x(W'_{y'}) \ni y $. By \cite[Theorem 4.1]{HcMcK}, for every $\delta \in \mathbb{Q}^+$ there exists a $\mathbb{Q}$-divisor $D'_y$ such that $V'_y$ is an exceptional lc centre and such that $D'_y \sim ((\lambda_x+1)(\mu_{y'}+1)-1+\delta)K_X$. Recall that $\lambda_x < \frac{3}{\alpha}$ and $\mu_{y'} < \frac{2}{\beta}$. Consider $J_0:=\left\{y \in \cup_{x \in Y_{2,a}}U'_x \textit{ s.t. } \dim(V'_y)=0\right\}$ and $J_1:=\left\{y \in \cup_{x \in Y_{2,a}}U'_x \textit{ s.t. } \dim(V'_y)=1\right\}$. Note that if $V'_y$ is a point, i.e. $V'_y=\{y\}$, then $LLC(X,D'_y,y)=\{y\}$, while if $V'_y$ is a curve then it is of general type because it passes through $y \in X_0$. Since $\cup_{x \in Y_{2,a}}U'_x$ is countably dense in $X$, then either $J_0$ or $J_1$ is countably dense. 

If $J_0$ is countably dense then we can apply Lemma \ref{tanti divisori} and Lemma \ref{nonvanishing} (with $m=n$) to conclude that, assuming $\epsilon, \delta$ very small and 
\begin{equation} \label{b1}
\beta > \frac{2}{1-\delta},
\end{equation}
for every $n \geq 1$ if 
\begin{equation}\label{a1}
\left(\frac{3}{\alpha}+1\right)(1+2/\beta)-1+ \delta \leq 1 \Leftrightarrow \alpha \geq \frac{3\beta+6}{\beta(1 - \delta)-2}
\end{equation} 
then $h^0((n+1)K_X) \geq n$.
If $J_1$ is countably dense then we can argue exactly in the same way as we did before for $Y_1$ countably dense: simply re-read the proof substituting $Y_1$ with $J_1$ and $\lambda_x$ with $(\lambda_x+1)(\mu_{y'}+1)-1+\delta$. We can conclude that, assuming  $\epsilon, \delta$ very small and 
\begin{equation} \label{b2}
\beta > \frac{2}{ \left(2-\delta \right) \left( \frac{2g-2-\epsilon}{2g-1-\epsilon}\right) -1-\delta }, 
\end{equation}
for every $n \geq 1$ if

\begin{align}
\left( \left( \left( \frac{3}{\alpha}+1 \right) \left(\frac{2}{\beta}+1 \right)-1 + \delta \right)+1 \right) \left(1+\frac{1}{2g-2-\epsilon} \right) -1 + \delta\leq1 \Leftrightarrow \\ 
\Leftrightarrow  \alpha \geq \frac{3\beta+6}{\beta \left( \left(2-\delta \right) \left( \frac{2g-2-\epsilon}{2g-1-\epsilon}\right) -1-\delta \right)-2} \label{a2}&
\end{align} 
then $h^0((n+1)K_X) \geq n$. 

Let us now suppose that $Y_{2,b}$ is countably dense.  Recall that for every $x \in Y_{2,b}$ we have a divisor $D_x \sim \lambda_x K_X$ such that $V_x=LLC(X, D_x, x)$ is an exceptional log canonical centre and $\dim(V_x)=2$. Since if we decompose a countably dense set as a countable union of subsets then at least one of the subsets is countably dense, we can suppose that $\lambda_x=\lambda$ for a fixed $\lambda \in \mathbb{Q}^+$. Recall that $\lambda < \frac{3}{\alpha}$. By \cite[Lemma 3.2]{McK}, \cite[Lemma 3.2]{Todorov}, we are in the following situation: 

\begin{displaymath}
\xymatrix{ X'  \ar[r] ^{\pi} \ar[d]^f & X \\
B &}
\end{displaymath}
 where $X'$, $B$ are normal projective varieties, $f$ is a dominant morphism with connected fibres, $\pi$ is a dominant and generically finite morphism and the image under $\pi$ of a general fibre of $f$ is $V_x$.
Arguing exactly as in \cite{Todorov} we can suppose that there exists a proper closed subset $X_1 \subset X$ such that for all $x \not \in X_1$, $D_x$ is smooth at $x$. %Hence also $V_x$, that has dimension $2$, is smooth at  $x$. 
Either $\pi$ is birational or the inverse image of a general $x \in X \setminus X_1$ under $\pi$ is contained in at least two different fibres of $f$.
% in fact if $\pi$ is not birational, since $\pi$ is a dominant generically finite morphism, then there exists an open set $O' \subseteq X'$ and $m \in \mathbb{N}$, $m > 1$, such that for every $x \in O'$ $\#\pi^{-1}(x)=m$; let $F_x \subset X'$ be a fibre of $f$ such that $\pi(F_x)=V_x$; $F_x$ and $V_x$ are birational through $\pi|_{F_x}$ by \cite{McK}, hence, since $V_x$ is smooth at $x$, %by Zariski's Main Theorem( cf. \cite{Hartshorne}
%$\pi|_{F_x}^{-1}(x)$ is connected, that is $\#(\pi^{-1}(x) \cap F_x)=1$. 
%Using the above fact and since by construction $B$ parametrizes log canonical centres 
%% il discorso sarebbe questo: siccome alla fin fine B si vede come sottoinsieme dello schema di hilbert, ogni punto di B corrisponde a sottovariet, sempre diverse. Ma sottovariet sempre diverse tramite \pi non possono andare nello stesso centro lc, essenzialmente perch i centri lc sono eccezionali.
%then if $\pi$ is not birational there are at least two log canonical centres through $x$. 
In the latter case we can apply \cite[Lemma 3.3]{Todorov} and Lemma \ref{tiebreaking}, \ref{due}., \ref{tre}., to conclude that there exists a countably dense set $Y:=Y_{2,b}\cap (X \setminus X_1)$ such that for all $y \in Y$ there exists a divisor $S_y \sim k(2\lambda K_X)$ ( $0 < k \leq 1$, $\lambda < \frac{3}{\alpha}$) such that $LLC(S_y, y)=\{C_y\}$, where $C_y$ is an irreducible variety of dimension at most $1$. Therefore, as in the case of $Y_0$ and $Y_1$, we can apply Lemma \ref{tanti divisori} and Lemma \ref{nonvanishing} (with $m=n$) to conclude that for every $n \geq 1$, if $2\lambda k < \frac{6}{\alpha} \leq 1$, that is 

\begin{equation} \label{a2bis}
\alpha \geq 6
\end{equation}
and
\begin{equation} \label{a3}
\left(\frac{6}{\alpha}+1\right)(1+1/(2g-2-\epsilon))-1+ \delta \leq 1 \Leftrightarrow \alpha \geq \frac{12g-6-6\epsilon}{(2g-2-\epsilon)(1 - \delta)-1}
\end{equation}
(we are considering $\epsilon, \delta$ very small) then $h^0((n+1)K_X) \geq n$.

We can now suppose that $\pi$ is birational. Again arguing as in \cite{Todorov}, we can suppose $X'=X$ and that the general fibre of $f$ over a point $b \in B$, $X_{b}$, is minimal and smooth (and of general type). Moreover for every $b \in B$ there exists a divisor $D_{b} \sim \lambda K_X$ for which we have $\mathcal{J}(D_{b}) \subset \mathcal{O}_X(-X_{b})$ (since the fibre is an exceptional lc centre for $(X, D_{b})$). Hence we are exactly in the situation of Proposition \ref{propositionfibration}: setting $k=1$, we know that if $\lambda (4(n+1)[\beta^2] -1)< 1$ and $n > \lambda$ then there is a surjection $H^0(X, \mathcal{O}_X((n+1)K_X)) \rightarrow H^0(X_b, \mathcal{O}_{X_b}((n+1)K_{X_b}))$ and thus the theorem is proved, because, by \cite[VII.5.4]{BPVdV}, $h^0(X_b, \mathcal{O}_{X_b}((n+1)K_{X_b})) \geq n$. Since $\lambda <\frac {3}{\alpha}$ then the numerical conditions are satisfied as long as

\begin{equation} \label{a4}
\alpha \geq 12(n+1) [\beta^2] -3
\end{equation}

\begin{equation} \label{a5}
\alpha \geq \frac{3}{n}
\end{equation}
\\

It is now time to put everything together, that is to find the best possible value for $\beta$ such that we have the lowest inferior bound for $\alpha$.

Set $b:= \frac{2g-3}{2g-1}$. Note that (\ref{b2}) implies (\ref{b1}) and that $\beta >\frac{2}{b}=\frac{4g-2}{2g-3}$ implies (\ref{b2}). Moreover (\ref{a3}) $\Rightarrow$ (\ref{a0}) $\Rightarrow$ (\ref{a-1}) $\Rightarrow$ (\ref{a5}), (\ref{a3}) $\Rightarrow$ (\ref{a2bis}) and (\ref{a2}) $\Rightarrow$ (\ref{a1}). Besides (\ref{a2}) $\Rightarrow$ (\ref{a3}) if $\beta < \frac{12g-10}{2g-3}$. Therefore we are left to consider only the conditions (\ref{a2}) and (\ref{a4}).  Set $\beta':=\frac{1+\sqrt{1+\frac{b(b+1)}{4(n+1)}}}{b}$ and, finally, choose $\beta:=\sqrt{[\beta'^2]+1-\epsilon'}$ with $0<\epsilon' \ll 1$ and such that $\beta \in \mathbb{Q}$. In this way $\beta > \beta' > \frac{2}{b}=\frac{4g-2}{2g-3}$, $[\beta^2]=[\beta'^2]$ and, actually, $[\beta^2]=4$ for every $n \geq 1$ as soon as $g \geq 19$.

Since (\ref{a2}) does not depend on $n$, for $n$ sufficiently large (\ref{a4}) $\Rightarrow$ (\ref{a2}). Moreover, with that choice of $\beta$ and with $g$ sufficiently large (namely $g \geq 23$), we have that (\ref{a4}) $\Rightarrow$ (\ref{a2}) for every $n$.  Therefore in general (\ref{a4}) $\Rightarrow$ (\ref{a2}), except for the following  finite number of couples $(g,n)$:

$g=2, n=1$;
 $g=4, n=7$;
$g=5, n=2$;
 $g=6, 44 \leq n \leq 52$;
 $g=7, n=2$;
 $g=10, 305 \leq n \leq 381$;
 $g=11, n=9,10$;
 $g=12, n=5$;
 $g=13, n=3$;
 $g=15, n=2$;
 $19 \leq g \leq 22, n=1$.

The theorem now follows by simple computations.

%Hence we can conclude that in general if $\alpha \geq 48(n+1)-3$ then $h^0((n+1)K_X)\geq n$, for all $n \geq 1$. In the following cases, however, we need to take other inequalities for $\alpha$:

%\begin{enumerate}[]
%\item $g=2$ : $n=1: \alpha \geq 879$, \ $ n \geq 2: \alpha \geq 432(n+1)-3 $;
%\item $g=3$ : $ n \geq 1: \alpha \geq 132(n+1)-3 $; 
%\item $g=4$ : $1 \leq n \leq 6: \alpha \geq 96(n+1)-3$, \ $ n=7: \alpha \geq 714 $, \ $n \geq 8: \alpha \geq 84(n+1)-3$;
%\item $g=5$: $n=1:\alpha \geq 84(n+1)-3=165$, \ $n=2: \alpha \geq 242$, \ $n \geq 3: \alpha \geq 72(n+1)-3$;
%\item $g=6$ : $1 \leq n \leq 43: \alpha \geq 72(n+1)-3$, \  $44 \leq n \leq 52$:  $\alpha \geq 3234$, \ $n \geq 53: \alpha \geq 60(n+1)-3$;
%\item $g=7$ : $n=1: \alpha \geq 72(n+1)-3=141$,\ $n=2: \alpha \geq 184$, \ $n\geq3: \alpha \geq 60(n+1)-3$;
%\item $g=8,9$: $n \geq 1: \alpha \geq 60(n+1)-3$;
%\item $g=10$ : $1 \leq n \leq 304: \alpha \geq 60(n+1)-3$, \ $ 305 \leq n \leq 381: \alpha \geq 18354$;
%\item $g=11$: $1 \leq n \leq 8 : \alpha \geq 60(n+1)-3$, \ $n=9,10: \alpha \geq 550$;
%\item $g=12$ : $1 \leq n \leq 4: \alpha \geq 60(n+1)-3$, \ $n=5: \alpha \geq 306$;
%\item $g=13$ : $n=1,2: \alpha \geq 60(n+1)-3$, \ $n=3: \alpha \geq 223$;
%\item $g=14$ : $n=1,2: \alpha \geq 60(n+1)-3$;
%\item $g=15$ : $n=1: \alpha \geq 60(n+1)-3=117$, \  $n=2: \alpha \geq 156$;
%\item $g=16,17,18$ : $n=1: \alpha \geq 60(n+1)-3=117$;
%\item $g=19$: $n=1: \alpha \geq 111$;
%\item $g=20$: $n=1: \alpha \geq 105$;
%\item $g=21$: $n=1: \alpha \geq 101$;
%\item $g=22$: $n=1: \alpha \geq 97$.
%\end{enumerate}
For the last statement just notice that if we go back over the above proof but using Lemma \ref{nonvanishing} with $n=1$ and $m=2$ (instead of $n=1$ and $m=1$) then we can conclude that $h^0(3K_X) > 0$ when $g=2, \alpha \geq 141$, or $g=3, \alpha \geq 69$, or $g=4, \alpha \geq 47$, or $g \geq 5, \alpha \geq 33$. Therefore, for $n=1$, if $g, \alpha$ are as in the hypotheses of the theorem  then not only $h^0(2K_X) > 0$ but also $h^0(3K_X)>0$ so, in these cases, we can say that $h^0(lK_X) \geq 1$ for every $l \geq 2$.

\end{proof}
\begin{remark} \label{controesempio}
There are examples of smooth threefolds $X$ with arbitrarily large volume but $h^0(K_X)=0$: in fact just choose a smooth surface of general type $S$ with $H^0(K_S)=0$, for example a numerical Godeaux surface (see \cite[VII, 10.1]{BPVdV}) and a smooth curve $C$ of genus $g$. Then set $X:=S \times C$:  $\vol(X)=3\vol(C) \vol(S)=3(2g-2)\vol(S) \xrightarrow[g \rightarrow +\infty]{} +\infty$, but by Kunneth's formula $H^0(K_{X}) \cong H^0(K_S) \otimes H^0(K_C)=0$.
\end{remark}
\section{Birationality of pluricanonical maps for threefolds of general type}
In order to have effective estimates on which pluricanonical system determines a birational map, by generic smoothness it suffices to understand when pluricanonical systems separate very general points.  Since we now need to keep track of two points and not only one, in this case to have the best results we cannot argue exactly in the same way as before (that is, applying \cite[Theorem 4.1]{HcMcK}).

Therefore, following \cite{Todorov}, we will use a slightly different technique by Takayama to inductively lower the dimension of lc centers on a birational modification of the original variety.

\begin{theorem} \label{birthreefold}
Let $X$ be a smooth, not $g$-countably dense, projective threefold of general type and such that $\vol(X) > \alpha^3$. Let $l \in \mathbb{N}$, $l \geq 5$. Let $$f(l,g):=3 \sqrt[3]{2} \left( 4l \left[ \frac{32g^2}{((g(l-1)-(l+1))^2}\right]-1 \right).$$ If $l,g, \alpha$ are as in Table \ref{tavolabirthreefold} or, in the other cases, $$\alpha >   \frac{3\sqrt[3]{2}g(1+ 2\sqrt{2})}{g(l-1)-(l+1)-4\sqrt{2}g}$$ then $|lK_X|$ gives a birational map. %$\alpha > 1917\sqrt[3]{2}$ (or, in case $g \geq 10$, $\alpha> 117\sqrt[3]{2}$)  . More generally, if  $g \not = 9$, $\alpha >  3\sqrt[3]{2}\left( 20\left[  \frac{8g^2}{(2g-3)^2}\right] -1\right)$ or $g=9$, $\alpha > 118\sqrt[3]{2}$ then $|lK_X|$ gives a birational map for every $l \geq5$.

\begin{table}[!ht] \caption{} \label{tavolabirthreefold}
\begin{center}
\begin{tabular}{c|c|c||c|c|c}
$l$ &$g$ & $\alpha$ & $l$ &$g$ & $\alpha$ \\ \cline{1-6}
$5$ & $\not= 9$ & $>f(l,g)$ & $ 9$ & $\leq 4$ & $> f(l,g) $ \\  \cline{2-3} \cline{4-6}
 & $9$ & $118\sqrt[3]{2}$ & $10$ & $\leq 3 $ & $> f(l,g)$ \\ \cline{1-3}  \cline{4-6}
$6$ & $\not= 8$ & $>f(l,g)$ & $ 11$ & $2$ & $> f(l,g)=261\sqrt[3]{2} $ \\  \cline{2-3} \cline{4-6}
 & $8$ & $73\sqrt[3]{2}$ & $12$ & $2 $ & $> f(l,g)=141 \sqrt[3]{2}$ \\ \cline{1-3}  \cline{4-6}
 $7$ & $\leq 23$ & $>f(l,g)$ & $ 13$ & $2$ & $> f(l,g)=153\sqrt[3]{2} $ \\  \cline{2-3} \cline{4-6}
 & $24,\ldots,39$ & $81\sqrt[3]{2}$ & $14$ & $2 $ & $> f(l,g)=165 \sqrt[3]{2}$ \\ \cline{1-3}  \cline{4-6}
 $8$ & $\leq 6$ & $>f(l,g)$ & $ $ & $ $ & $ $ \\  \cline{2-3}
 & $7$ & $93\sqrt[3]{2}$ & $ $ & $ $ & $ $ \\ \cline{1-3}  \cline{4-6}
\end{tabular}
\end{center}
\end{table}

\end{theorem}

\begin{corollary}
If $\alpha > 1917\sqrt[3]{2}$ (or, in case $g \geq 10$, $\alpha> 117\sqrt[3]{2}$)  then $|lK_X|$ gives a birational map for every $l \geq 5$. More generally, if  $g \not = 9$, $\alpha >  3\sqrt[3]{2}\left( 20\left[  \frac{8g^2}{(2g-3)^2}\right] -1\right)$ or $g=9$, $\alpha > 118\sqrt[3]{2}$ then $|lK_X|$ gives a birational map for every $l \geq 5$. 
\end{corollary}

\begin{remark}
This improves \cite[Theorem 1.2]{Todorov}.
\end{remark}

\begin{remark}
By \cite{ChenChenII} we know that if $l \geq 73$ then $|lK_X|$ is always birational, independently of the volume of $X$. 
\end{remark}

\begin{proof}
We will follow \cite{Todorov}. By \cite[Theorem 3.1]{Takayama}, for every $ 0 < \epsilon < 1$ there exists a smooth projective variety $X'$, a birational morphism $\mu: X' \rightarrow X$ and an approximate Zariski decomposition $\mu^*(K_X) \sim_{\mathbb{Q}} A+E$ where $A=A_{\epsilon}$ is an ample $\mathbb{Q}$-divisor and $E=E_\epsilon$ is an effective $\mathbb{Q}$-divisor that satisfy condition (1),(2),(3) of Takayama's theorem (cf. \cite[Theorem 3.1]{Takayama}).

Without loss of generality we can argue on $X'$ instead of $X$
By Remark \ref{generaltype} $X'$ is at least not $2$-countably dense, hence $g\geq 2$. Furthermore, since $X'$ is not $g$-countably dense then by Remark \ref{notcountablydense} there exists a very general subset $\Lambda \subseteq X'$ such that every curve passing through any point of $\Lambda$ has geometric genus $ \geq g$.
Now we would like to simply apply \cite[Proposition 5.3]{Takayama}, but in order to have better numerical conditions, as in the proof of Theorem \ref{threefold} we will distinguish two different cases depending on the volume of lc centres.

By \cite[Lemma 5.4]{Takayama}, there exists a very general subset $U$ of $X'$ such that for every two distinct points $x, y \in U$ we can construct, depending on $x,y$, an effective divisor $D_1 \sim_{\mathbb{Q}} a_1 A$, with $a_1 < \frac{3 \sqrt[3]{2}}{\alpha (1-\epsilon)}, a_1 \in \mathbb{Q}^+$, such that $x,y \in Z(\mathcal{J}(D_1))$,$(X',D_1)$ is lc not klt at one of the points, say $p(x,y) \in \{x,y\}$, and either $\codim Z(\mathcal{J}(D_1)) > 1$ at $p(x,y)$ or there is one irreducible component of $Z(\mathcal{J}(D_1))$, say $V_{p(x,y)}$, that passes through $p(x,y)$ and such that $\codim V_{p(x,y)} = 1$. We can suppose $U \subseteq \Lambda$.

Fix $\beta \in \mathbb{Q}^+$. 

Let $U':=\{p(x,y) | \codim Z(\mathcal{J}(D_1))=1  \text{ at } p(x,y) \text{ and } \vol(V_{p(x,y)}) \leq \beta^2\}$.  Since $U=U' \cup (U \setminus U')$, then by Lemma \ref{general set}, 4., we are in one of this two cases:

%Let $U':=\{x \in U | \forall y \in U \ \codim Z(\mathcal{J}(D_1))>1  \text{ at } p(x,y) \text{, or } \vol(V_{p(x,y)}) > \beta^2 \}$ and let $U'':=\{p(x,y) | \codim Z(\mathcal{J}(D_1))=1  \text{ at } p(x,y) \text{ and } \vol(V_{p(x,y)}) \leq \beta^2\}$.  We are in one of this two cases:

\begin{enumerate}
\item $U \setminus U'$ contains a very general subset $U''$ of $X$;
\item $U'$ is countably dense.
\end{enumerate}

In the first case we know that $\forall x,y \in U''$ either $\codim Z(\mathcal{J}(D_1))>1$ at $p(x,y)$ or $\vol(V_{p(x,y)}) > \beta^2$. Applying the inductive steps of \cite[Proposition 5.3]{Takayama}, we can conclude that given two very general points $x,y \in X'$ there exists (depending on $x,y$) an effective $\mathbb{Q}$-divisor $D$ on $X'$ and $a \in \mathbb{Q}^+$ with $D \sim_\mathbb{Q} aA$ such that $x,y \in Z(\mathcal{J}(X', D) )$ with $\dim Z(\mathcal{J}(X', D) )=0$ around $x$ or $y$, that is $x$ or $y$ is an isolated point of $Z(\mathcal{J}(X', D) )$, and $$a < \left( 1+ \frac{1}{(1-\epsilon)(g-1)} \right) \left( 1+\frac{2 \sqrt{2}}{(1-\epsilon)\beta}\right) \left( 2+ \frac{3 \sqrt[3]{2}}{(1-\epsilon)\alpha}\right)-2+\epsilon f,$$
where $f= \left( 1+ \frac{1}{(1-\epsilon)(g-1)} \right) \left( 2+\frac{2 \sqrt{2}}{(1-\epsilon)\beta}\right)>0$.

By \cite[1.41]{Debarre}, $K_{X'} \sim_{\mathbb{Z}} \mu^*(K_X) + \text{Exc}(\mu) \sim_{\mathbb{Q}}A+E+\text{Exc}(\mu)$, where $\text{Exc}(\mu)$ is the exceptional locus and it is an effective divisor by \cite[ 1.40]{Debarre}. Therefore, replacing $D$ with $D+(l - 1)(E+\text{Exc}(\mu))$ (with $l \in \mathbb{N}^+$), as in the proof of Lemma \ref{nonvanishing} we can conclude that by Nadel's vanishing theorem $|l K_{X'}|$ separates two very general points in $X'$ as soon as $l \geq [a]  + 2$.

Hence, in the first case, considering $l \geq 5$, we now need only to estimate $\alpha$ (depending on $g,\beta,\epsilon$) in order to have $[a] \leq l-2$, that is $a < l-1$. To that purpose, choosing $\epsilon$ sufficiently small and 
\begin{equation}
\beta > \frac{4\sqrt{2}g}{g(l-1)-(l+1)} \label{birbeta1},
\end{equation}
it is enough to ask that

\begin{equation} \nonumber
\left( 1+ \frac{1}{g-1} \right) \left( 1+\frac{2 \sqrt{2}}{\beta}\right) \left( 2+ \frac{3 \sqrt[3]{2}}{\alpha}\right) < l+1  \\
\end{equation}
\begin{equation} \label{biralpha1}
\Leftrightarrow \alpha> 
\frac{3\sqrt[3]{2}g(\beta + 2\sqrt{2})}{\beta(g(l-1)-(l+1))-4\sqrt{2}g}.
\end{equation} 

If, otherwise, the second case occur then $\beta \geq 1$ since the volume of a surface of general type is at least $1$. Moreover by \cite[Lemma 3.2]{McK}, we are in the following situation: 

\begin{displaymath}
\xymatrix{ X''  \ar[r] ^{\pi} \ar[d]^f & X' \\
B &}
\end{displaymath}
 where $X''$, $B$ are normal projective varieties, $f$ is a dominant morphism with connected fibres, $\pi$ is a dominant and generically finite morphism and the image under $\pi$ of a general fibre of $f$ is $V_x$ where $V_x$ is a surface through $x$, a general point. Moreover there exists a divisor $D_x$ such that $V_x$ is a pure log canonical centre of $(X', D_{x})$. In addition, setting $\overline{a}:=  \frac{3\sqrt[3]{2}}{\alpha(1-\epsilon)} \in \mathbb{Q}^+ $ we have that $D_x \sim_{\mathbb{Q}} \overline{a}A$. Moreover for every $p$ in a countably dense subset of $U'$, $V_{p}$ is the image through $\pi$ of a fiber of $f$. 
 
Again as in \cite{Todorov}, we can suppose that there exists a proper closed subset $X'_1 \subset X'$ such that for all $x \not \in X'_1$, $D_{x}$ is smooth at $x$. 

As in \cite{Todorov} and the proof of Theorem \ref{threefold}, we will distinguish two other different subcases, depending on the birationality of $\pi$: in fact, as we have already proved in Theorem \ref{threefold}, either $\pi$ is birational or for a general $x \in X' \setminus X'_1$ there are at least two log canonical centres through $x$. In the latter case for every $x,y$ general points in $X'$ we can produce an effective $\mathbb{Q}$-divisor $D_{x,y}$ on $X'$ and a positive rational number (depending on $x,y$) $\overline{a}'' < \epsilon+ \frac{9\sqrt[3]{2}}{\alpha(1-\epsilon)}$ such that $D_{x,y} \sim_\mathbb{Q} \overline{a}''A$ and such that $D_{x,y}$ satisfies the induction statement $(*_2)$ of \cite[Proposition 5.3]{Takayama}.   

We can now apply the inductive steps of Takayama (see \cite[Proposition 5.3; in particular Lemmas 5.5, 5.8]{Takayama}) and conclude that for every $x,y$ general points in $X'$ there exists an effective $\mathbb{Q}$-divisor $D'_{x,y}$ on $X'$ and a positive rational number $\overline{a}'''$ such that  $D'_{x,y}\sim_{\mathbb{Q}}\overline{a}'''A$, $x,y \in Z(\mathcal{J}(X', D'_{x,y}))$ with $\dim Z(\mathcal{J}(X', D'_{x,y}))=0$ around $x$ or $y$ and  $$\overline{a}''' < \left( 2 + \frac{2}{(1-\epsilon)(g-1)} \right) \left( 1+ \frac{\frac{9}{2}\sqrt[3]{2}}{(1-\epsilon)\alpha}\right)-2+2\epsilon h,$$
where $h= \left( 1+ \frac{1}{(1-\epsilon)(g-1)} \right)>0$. 

As before, we conclude that $|lK_X'|$ separates $x$ and $y$ as soon as $\overline{a}''' < l-1$. To that purpose, choosing $\epsilon$ sufficiently small, it is enough to ask that 
\begin{equation} \label{biralpha3}
\left( 2+\frac{2}{g-1} \right) \left( 1+ \frac{\frac{9}{2}\sqrt[3]{2}}{\alpha} \right) < l+1 \Leftrightarrow \alpha >  \frac{9 \sqrt[3]{2}g}{g(l-1)-(l+1)}.
\end{equation}

We can now assume that $\pi$ is birational. Moreover since $K_{X''} \sim \pi^*(K_X')+ \textrm{Exc}(\pi)$, then we can suppose that the general fiber $X''_b$ is a pure log canonical centre of $D''_b \sim_{\mathbb{Q}} \overline{a} K_{X''}$.  Arguing as in \cite{Todorov} we can suppose $X''$ is smooth and that the general fiber $X''_b$ of $f$ is smooth and minimal (and of general type). In addition, since the fibers of $f$ are all numerically equivalent, we also know that $\vol(X''_b) \leq \beta^2$.

As before to prove that $|lK_{X'}|$ separates two very general points it is enough to show that $|lK_{X''}|$ separates two very general points on $X''$. 

Choose $x,y$ general points on the same fiber. Since for all $l \geq 5$, $|lK_{X''_b}|$ gives a birational map on $X''_b$ by a result of Bombieri (cf.\ \cite{Bombieri}), in order to separate $x$ and $y$ we can simply apply Proposition \ref{propositionfibration} with $k=1$, $n=l-1$, obtaining the following conditions:
\begin{equation}
\overline{a}(4l[\beta^2]-1)<1,
\end{equation}
\begin{equation}
\overline{a} < l-1,
\end{equation}
that are implied by
\begin{equation} \label{biralpha4}
\alpha > 3 \sqrt[3]{2}(4l[\beta^2]-1),
\end{equation}
\begin{equation} \label{biralpha5}
\alpha > \frac{3\sqrt[3]{2}}{l-1}.
\end{equation}
(Recall that $\beta \geq 1$ and  hence $[\beta^2] \geq 1$).

If $x,y$ are on different fibers then, since by a result of Bombieri $H^0(2K_{X''_b}) \not = 0$, we can apply Proposition \ref{propositionfibration} with $k=2$ and $n=1$, obtaining the following conditions:

\begin{equation}
2\overline{a}(8[\beta^2]-1)<1,
\end{equation}
\begin{equation}
\overline{a} < \frac{1}{2},
\end{equation}
that are implied by
\begin{equation} \label{biralpha6}
\alpha > 6 \sqrt[3]{2}(8[\beta^2]-1),
\end{equation}
\begin{equation} \label{biralpha7}
\alpha > 6\sqrt[3]{2}.
\end{equation}

Under these assumptions $H^0(2K_{X''})$ separates $x$ and $y$. Then if $l$ is even also $H^0(lK_{X''})$ separates $x$ and $y$. If $l$ is odd then if moreover
\begin{equation} \label{biralpha8}
H^0(3K_{X''}) \not = 0
\end{equation}
we can conclude in the same way. 

%Note that if $x,y$ are on different fibers we could have argued also in another, symmetric way: by \cite[VII.5.4]{BPVdV}, $H^0(3K_{X''_b}) \not = 0$, thus we could have applied proposition \ref{propositionfibration} with $k=2$, as before, but with $n=2$. But as a matter of fact, in this case we would not obtain numerical improvements.

%If $x,y$ are on different fibers we can argue also in another, symmetric way: by \cite{BPVdV}, VII.5.4, $H^0(3K_{X''_b}) \not = 0$, thus we can apply proposition \ref{propositionfibration} with $k=2$, as before, but with $n=2$. In this case we obtain the following conditions:
%\begin{equation}
%2\overline{a}(12[\beta^2]-1)<1,
%\end{equation}
%\begin{equation}
%\overline{a} < 1,
%\end{equation}
%that are implied by
%\begin{equation} \label{biralpha9}
%\alpha > 6 \sqrt[3]{2}(12 [\beta^2]-1),
%\end{equation}
%\begin{equation} \label{biralpha10}
%\alpha > 3\sqrt[3]{2}.
%\end{equation}
%Under these assumptions $H^0(3K_{X''})$ separates $x$ and $y$. Then if
%\begin{equation} \label{biralpha11}
%H^0(2K_{X''}) \not = 0
%\end{equation}
%also $H^0(5K_{X''})$ separates $x$ and $y$. 

To deal with condition (\ref{biralpha8}) we could simply use Theorem \ref{threefold}, but since we do not need $h^0(3K_{X''}) \geq 2$ (because for our purposes it is enough to ask that $h^0(3K_{X''}) \geq 1$) then instead of applying Theorem \ref{threefold} in its full extent we can simply use the results about $h^0(3K_X)$ stated at the end of the proof of Theorem \ref{threefold}.

It is now time to put everything together, that is to find the best possible value for $\beta$ such that we have the lowest inferior bound for $\alpha$.

%Since $g \geq 2$, then by (\ref{birbeta1}), and for any $g$, $\beta > \sqrt{2}$. Therefore (\ref{biralpha9}) $\Rightarrow$ (\ref{biralpha4}) $\Rightarrow$ (\ref{biralpha3}) $\Rightarrow$ (\ref{biralpha2}) $\Rightarrow$ (\ref{biralpha5}). Moreover 
%(\ref{biralpha4}) $\Rightarrow$ (\ref{biralpha6}) $\Rightarrow$ (\ref{biralpha7})  $\Rightarrow$ (\ref{biralpha10}).

%Thus we are left to consider only two distinct sets of conditions: (\ref{birbeta1}), (\ref{biralpha1}), (\ref{biralpha4}), (\ref{biralpha8}) or (\ref{birbeta1}), (\ref{biralpha1}), (\ref{biralpha9}), (\ref{biralpha11}). 

%Set $\beta':=\frac{2\sqrt{2}g}{2g-3}$ and, finally, choose $\beta:=\sqrt{[\beta'^2]+1-\epsilon'}$ with $0<\epsilon' \ll 1$ and such that $\beta \in \mathbb{Q}$. (\ref{birbeta1}) is obviously verified. Besides, in this way $[\beta^2]=[\beta'^2]$ and, actually, $[\beta^2]=2$ as soon as $g \geq 9$.

%We can now conclude, by simple computations and using theorem \ref{threefold}. Note that for $g=2, \ldots, 8$ and $g \geq 15$ it is better to use the first set of conditions, while for  $g=9, \ldots, 14$ the second set will do.

If $\beta < 1$ we need only to consider (\ref{birbeta1}) and (\ref{biralpha1}). 

If $\beta \geq 1$ then, since $l \geq 5$ and $g \geq 2$, (\ref{biralpha4}) $\Rightarrow$ (\ref{biralpha3}) $\Rightarrow$ (\ref{biralpha5}) and (\ref{biralpha4}) $\Rightarrow$ (\ref{biralpha6}) $\Rightarrow$ (\ref{biralpha7}). Moreover, by (\ref{birbeta1}), (\ref{biralpha4}) $\Rightarrow$ (\ref{biralpha8}). %vedi quaderno 5 2/12/2009
Thus, if $\beta \geq 1$, we are left to consider only these conditions: (\ref{birbeta1}), (\ref{biralpha1}), (\ref{biralpha4}). 

Set $\beta':=\frac{4g\sqrt{2}}{g(l-1)-(l+1)}$. Since $l \geq 5$, $\beta'>0$. Finally, if $\beta' \geq 1$ choose $\beta:=\sqrt{[\beta'^2]+1-\epsilon'}$, if $\beta' < 1$ choose $\beta:=1-\epsilon'$, with $0<\epsilon' \ll 1$ and such that $\beta \in \mathbb{Q}$. (\ref{birbeta1}) is obviously verified. Besides, in this way $[\beta^2]=[\beta'^2]$. %and, actually, $[\beta^2]=2$ as soon as $g \geq 9$.

Now some simple computations allow us to conclude: for $(l,g)$ not as in Table \ref{tavolabirthreefold}  we have that $\beta' < 1$ and hence that $|lK_X|$ gives a birational map for $\alpha >   \frac{3\sqrt[3]{2}g(1+ 2\sqrt{2})}{g(l-1)-(l+1)-4\sqrt{2}g}$. For $l=7, g=24, \ldots, 39$ and $l=8$, $g=7$ it turns out that it is better to take a larger value for $\beta'$, namely $\beta'=1$. Hence for all $(l,g)$ as in Table  \ref{tavolabirthreefold} we have that (\ref{biralpha4}) $\Rightarrow$ (\ref{biralpha1}) except for $l=5, g=9$ and $l=6, g=8$.

%As for higher pluricanonical maps, just go back over the proof just seen, but with $c=5$ instead of $c=4$.  In this way we can prove that $|6K_X|$ gives a birational map in the following cases: $g \not = 8$, $\alpha >  3\sqrt[3]{2}\left( 24\left[ \frac{32g^2}{(5g-7)^2}\right] -1\right)$;  $g=8$, $\alpha > 73\sqrt[3]{2}$. Hence we can conclude: in fact, when $\alpha, g$ are as in the hypotheses of the theorem, we have that both $|5K_X|$ and $|6K_X|$ give a birational map and that $h^0(2K_X) \geq 1$ by theorem \ref{threefold}.
\end{proof}

\begin{remark} \label{4KX}
As Todorov pointed out in \cite{Todorov}, we cannot expect to have analogous results about $|4K_X|$. In fact just choose a smooth surface of general type $S$ such that $|4 K_S|$ does not give a birational map, for example a smooth minimal surface $S$ with $K_{S}^2=1$ and $h^0(K_S)=2$  (cf. \cite[VII, 7.1]{BPVdV}), and a smooth curve $C$ of genus $g$. Then set $X:=S \times C$.  $\vol(X)=3(2g-2)\vol(S) \xrightarrow[g \rightarrow +\infty]{} +\infty$, but since the map $\phi_{|4K_X|}$ given by $|4K_X|$ is, by Kunneth's formula, essentially constructed with the two maps $\phi_{|4K_S|}$ and $\phi_{|4K_C|}$ followed by a Segre's embedding, then $\phi_{|4K_X|}$ is never birational. 
\end{remark}

In the wake of Remark \ref{4KX}, and when the volume is sufficiently large, we can characterize threefolds for which $|4K_X|$ does not give a birational map. In fact if a threefold $X$ satisfies the conditions on $\alpha$ as listed in the proof of Theorem \ref{birthreefold} (imposing, this time, $l=4$) but, at the same time, $\phi_{|4K_X|}$ is not birational, then $X$ must necessarily be birational to a threefold fibered by surfaces for which the fourth pluricanonical map is not birational. Such surfaces $X''_b$ have volume $1$ and geometric genus $p_g=2$ by \cite{Bombieri} (see also\cite[Proposition VII.7.1 and VII.7.3]{BPVdV}). Therefore we can state the following:

\begin{corollary} \label{4thmap} Let $X$ be a smooth projective threefold of general type such that $\vol(X) > \alpha^3$. If $\alpha > 6141\sqrt[3]{2}$ %(or, in case $g \geq 42$, $\alpha > 141 \sqrt[3]{2}$)
 then $|4K_X|$ does not give a birational map if, and only if, $X$ is birational to a fibre space $X''$, with $f: X'' \rightarrow B$, where $B$ is a curve, such that the general fiber $X''_b$ is a smooth minimal surface of general type with volume $1$ and geometric genus $p_g=2$. More generally, if $X$ is not $g$-countably dense and if $g,\alpha$ are as in Table \ref{tavola4thmap} or, in the other cases, $\alpha > 3 \sqrt[3]{2} \left( 16 \left[ \frac{32g^2}{(3g-5)^2}\right] -1\right)$, then $|4K_X|$ does not give a birational map if, and only if, $X$ is birational to a fibre space $X''$ as above. 
\begin{table}[!ht] \caption{} \label{tavola4thmap}
\begin{center}

\begin{tabular}{c|c||c|c}
$g$ & $\alpha$ & $g$ & $\alpha$ \\ \hline
$11$ & $> 237\sqrt[3]{2}$ & $39$ & $> 168\sqrt[3]{2}$ \\ \hline
$30, \ldots, 37$ & $> 189\sqrt[3]{2}$ & $40$ & $ > 156 \sqrt[3]{2}$ \\ \hline
$38$ & $> 182\sqrt[3]{2}$ & $41$ & $ > 146 \sqrt[3]{2}$ \\ \hline
\end{tabular}
\end{center}
\end{table}
\end{corollary}

\begin{proof}
The ``if'' part is trivial (and not depending on $g, \alpha$). %se |4K| fosse birazionale allora anche 4K_S sarebbe birazionale, con S fibra-superficie. Assurdo per la caratterizzazione delle superfici con 4-canonical map birazionale
For the ``only if'', simply consider again all the conditions on $\alpha$ as in the proof of Theorem \ref{birthreefold}, but with $l=4$ instead of $l \geq 5$; moreover, instead of the usual value for $\beta'$, it is better to take a larger value in some cases: for $g=11$ $\beta':=\sqrt{5}$, for $g=30, \ldots, 37$, $\beta':=2$. Note also that the condition (\ref{biralpha8}) is not needed.
Now,  for $g=2, \ldots 37$ and $g \geq 42$ we have that (\ref{biralpha4}) $\Rightarrow$ (\ref{biralpha1}), while for $g=38,39,40,41$ (\ref{biralpha1}) $\Rightarrow$ (\ref{biralpha4}). 
\end{proof}

\begin{remark} \label{controesempiomappa4canonica}
In \cite{Dong} and \cite{ChenZhang} there is an example of a smooth canonical threefold $X$ with volume $=2$ and such that $|4K_X|$ does not give a birational map. For this $X$ the thesis of Corollary \ref{4thmap} does not apply: in fact for a generic irreducible curve $C_0$ in any family of curves on $X$ we have $K_X \cdot C_0 \geq 2$ (see \cite[Example 6.3]{ChenZhang}), but if $X$ were birationally fibred by surfaces of volume $1$ and $p_g=2$, we would have, on a general fibre, a family of curves for which $K_X \cdot C_0 \leq 1$.  %prendere curve nel sistema canonico.
\end{remark}

Analogously, dealing this time with the 3rd pluricanonical map, considering the characterization of surfaces with a birational 3rd pluricanonical map (cf.  \cite[Proposition VII.7.1, VII.7.2 and VII.7.3]{BPVdV}) and requiring $X$ not to be  $3$-countably dense, we can state also the following:

\begin{corollary} \label{3rdmap} Let $X$ be a smooth, not $3$-countably dense, projective threefold of general type such that $\vol(X) > \alpha^3$. If $\alpha > 5178\sqrt[3]{2}$ %(or, in case $g \geq 39$, $\alpha > 570 \sqrt[3]{2}$) 
then $|3K_X|$ does not give a birational map if, and only if, $X$ is birational to a fibre space $X''$, with $f: X'' \rightarrow B$, where $B$ is a curve, such that the general fibre $X''_b$ is a smooth minimal surface of general type and either it has volume $1$ and geometric genus $p_g=2$ or it has volume $2$ and $p_g=3$.  More generally, if $X$ is not $g$-countably dense, with $g \geq 3$, and if $g,\alpha$ are as in Table \ref{tavola3rdmap} or, in the other cases, $\alpha > 6 \sqrt[3]{2} \left( 12 \left[ \frac{8g^2}{(g-2)^2}\right] -1\right)$ then $|3K_X|$ does not give a birational map if, and only if, $X$ is birational to a fibre space $X''$ as above. 

\begin{table}[!ht] \caption{} \label{tavola3rdmap}
\begin{center}

\begin{tabular}{c|c||c|c}
$g$ & $\alpha$ & $g$ & $\alpha$ \\ \hline
$11$ & $> 858\sqrt[3]{2}$ & $35, \ldots, 37$ & $> 642\sqrt[3]{2}$ \\ \hline
$19$ & $> 714\sqrt[3]{2}$ & $38$ & $ > 640 \sqrt[3]{2}$ \\ \hline
\end{tabular}
\end{center}
\end{table}
\end{corollary}

\begin{proof}
Consider again the proof of Theorem \ref{birthreefold}, but with $l=3$: this time, however, if $x,y$ are on different fibers then we need to apply Proposition \ref{propositionfibration} with $k=2, n=2$ obtaining a new condition (\ref{biralpha6}), namely $\alpha> 6\sqrt[3]{2}(12[\beta^2]-1)$, and a new condition (\ref{biralpha7}), namely $\alpha> 3 \sqrt[3]{2}$. As before, (\ref{biralpha8}) is no longer needed. Moreover instead of the usual value for $\beta'$, it is better to take a larger value in some cases: for $g=11$ $\beta':=\sqrt{12}$, for $g=19$ $\beta':=\sqrt{10}$, for $g=35, \ldots, 37$, $\beta':=3$. Therefore this time we have that (\ref{biralpha6})$\Rightarrow$ (\ref{biralpha4}) and we are left to consider only conditions (\ref{birbeta1}), (\ref{biralpha1}) and (\ref{biralpha6}). For $g\not=38$  we have that (\ref{biralpha6}) $\Rightarrow$ (\ref{biralpha1}), while for $g=38$ (\ref{biralpha1}) $\Rightarrow$ (\ref{biralpha6}). 
\end{proof}

\begin{remark}
There are examples of threefolds $X$ of general type with large volume and $|3K_X|$ birational even if $X$ is covered by curves of genus $2$: just consider the product $C_2 \times C_g \times C_g$ (where $C_a$ is a smooth curve of genus $a$) and let $g$ go to infinity. 
\end{remark}

\begin{remark}
By Corollaries \ref{4thmap} and \ref{3rdmap} we have that if $X$ is a threefold of general type, not $3$-countably dense and of sufficiently large volume then the birationality of $|3K_X|$ implies  the birationality of $|4K_X|$.
\end{remark}

We can say something also for the second pluricanonical map, even if in this case we need to suppose that $X$ is not $4$-countably dense. Note that the classification of surfaces for which the second pluricanonical map is not birational has not been completed yet. The reader can refer to \cite[\S 2]{Catanese} for a survey on this subject and to \cite[Theorem 0.7, Remark 0.8]{Borrelli} for a partial classification (however notice that by our assumption about countably density the \textit{standard case} and the symmetric product case cannot occur).

\begin{corollary}  \label{2ndmap} Let $X$ be a smooth, not $4$-countably dense, projective threefold of general type and such that $\vol(X) > \alpha^3$. If $\alpha > 24570\sqrt[3]{2}$ %(or, in case $g \geq 244$, $\alpha > 1530 \sqrt[3]{2}$) 
then $|2K_X|$ does not give a birational map if, and only if, $X$ is birational to a fibre space $X''$, with $f: X'' \rightarrow B$, where $B$ is a curve, such that such that the general fiber $X''_b$ is a smooth minimal surface of general type and $|2K_{X''_b}|$ does not give a birational map.
More generally, if $X$ is not $g$-countably dense, with $g \geq 4$, and if $g,\alpha$ are as in Table \ref{tavola2ndmap} or, in the other cases, $\alpha > 6 \sqrt[3]{2} \left( 8 \left[ \frac{32g^2}{(g-3)^2}\right] -1\right)$, then $|2K_X|$ does not give a birational map if, and only if, $X$ is birational to a fibre space $X''$ as above. 
\begin{table}[!ht] \caption{} \label{tavola2ndmap}
\begin{center}
%\begin{corollary}  \label{2ndmap} Let $X$ be a smooth projective threefold of general type and such that $\vol(X) > \alpha^3$. Let $g$ be a positive integer and let $\Lambda$ be a very general subset of $X$ such that if $x \in \Lambda$ then every (possibly singular) curve through $x$ has geometric genus $\geq g$. Assume $g \geq 4$. If $\alpha > 24570\sqrt[3]{2}$ (or, in case $g \geq 244$, $\alpha > 1530 \sqrt[3]{2}$) then $|2K_X|$ does not give a birational map if, and only if, $X$ is birational to a fibre space $X''$, with $f: X'' \rightarrow B$, where $B$ is a curve, such that   the general fiber $X''_b$ is in one of the following cases:

%\begin{enumerate}
%\item $p_g=6$ and $\vol(X''_b)=8,9$ or $p_g=5$ and $\vol(X''_b)=7,8$ or $p_g=4$ and $\vol(X''_b)=6,7,8$. In all these cases $q=0$;
%\item $p_g=3, q=0$ and $\vol(X''_b)=5, 6,7$;
%\item $p_g=3$ and $q=3$. In this case $X$ is the symmetric product of a curve of genus $3$.
%\end{enumerate}
%More generally if $g,\alpha$ are as in the table below or, in the other cases, $\alpha > 6 \sqrt[3]{2} \left( 8 \left[ \frac{32g^2}{(g-3)^2}\right] -1\right)$, then $|2K_X|$ does not give a birational map if, and only if, $X$ is birational to a fibre space $X''$ as above. 

\begin{tabular}{c|c||c|c}
$g$ & $\alpha$ & $g$ & $\alpha$ \\ \hline
$8$ & $> 3930\sqrt[3]{2}$ & $43,44$ & $> 1770\sqrt[3]{2}$ \\ \hline
$12$ & $> 2730\sqrt[3]{2}$ & $53,54$ & $> 1722\sqrt[3]{2}$ \\ \hline
$14$ & $> 2490\sqrt[3]{2}$ & $69, \ldots, 72$ & $ > 1674 \sqrt[3]{2}$ \\ \hline
$22$ & $> 2058\sqrt[3]{2}$ &$73$ & $> 1630\sqrt[3]{2}$ \\ \hline
$24$ & $> 2010\sqrt[3]{2}$ &  $101, \ldots, 110$ & $ > 1626 \sqrt[3]{2}$ \\ \hline
$26$ & $> 1962\sqrt[3]{2}$ & $197, \ldots, 241$ & $> 1578\sqrt[3]{2}$ \\ \hline
$29$ & $> 1914\sqrt[3]{2}$ & $242$ & $ > 1560 \sqrt[3]{2}$ \\ \hline
$32$ & $> 1866\sqrt[3]{2}$ & $243$ & $> 1532\sqrt[3]{2}$ \\ \hline
$37$ & $> 1818\sqrt[3]{2}$ & & \\ \hline

\end{tabular}
\end{center}
\end{table}
\end{corollary}

\begin{proof}
Consider  the proof of Theorem \ref{birthreefold}, but with $l=2$ instead of $l \leq 5$. (\ref{biralpha8}) is not needed. Moreover instead of the usual value for $\beta'$, it is better to take a larger value in some cases: $g=8$ $\beta':=\sqrt{82}$, $g=12$ $\beta':=\sqrt{57}$, $g=14$ $\beta':=\sqrt{52}$, $g=22$ $\beta':=\sqrt{43}$, $g=24$ $\beta':=\sqrt{42}$, $g=26$ $\beta':=\sqrt{41}$, $g=29$ $\beta':=\sqrt{40}$, $g=32$ $\beta':=\sqrt{39}$, $g=37$ $\beta':=\sqrt{38}$, $g=43,44$ $\beta':=\sqrt{37}$, $g=53,54$ $\beta':=6$, $g=69, \ldots, 72$ $\beta':=\sqrt{35}$, $g=101, \ldots, 110$ $\beta':=\sqrt{34}$, $g=197, \ldots, 241$ $\beta':=\sqrt{33}$. This time we have that (\ref{biralpha6})$\Rightarrow$ (\ref{biralpha4}) and we are left to consider only conditions (\ref{birbeta1}), (\ref{biralpha1}) and (\ref{biralpha6}). For $g\not=73,242,243$  we have that (\ref{biralpha6}) $\Rightarrow$ (\ref{biralpha1}), while for $g=73,242,243$ (\ref{biralpha1}) $\Rightarrow$ (\ref{biralpha6}). 
\end{proof}

\begin{remark}
The birationality of the fourth pluricanonical map for threefolds of general type has been studied by, among the others, Lee, Dong, M.Chen, Zhang. Actually it is still an open problem when $\phi_{|4K_X|}$ not birational implies that $X$ is birational to an $X''$ as in Corollary \ref{4thmap} (cf. \cite[6.4]{ChenZhang}). 

As for the birationality of the third pluricanonical map, explicit characterizations not depending on the volume are not known (cf.\ \cite[Open problems 6.4]{ChenZhang}).
\end{remark}

\section{Higher dimensional results}

We know that there exists a positive lower bound on the volume of any variety of general type of a given dimension (see, for example, \cite[Theorem 1.2]{Takayama}). If only we knew these lower bounds explicitly then the ideas we exploited for threefolds to find estimates for the non-vanishing of pluricanonical systems or the birationality of pluricanonical maps could be generalized to varieties of any dimension. Unfortunately this is not the case. Anyway, we did explicit calculations in the case of fourfolds, since in \cite{ChenChenII} J. Chen and M. Chen computed a lower bound for the volume for threefolds of general type. However notice that since we do not have the technique of the fibration at our disposal, these estimates are far from being optimal. 

\begin{theorem}  \label{dimensionealta}
Let $X$ be a smooth projective variety of general type and of dimension $d$, such that $\vol(X) > \alpha^d$. Let $\Pi$ be a very general subset of $X$ and, for $i=1, \ldots, d-1$, let $v_i \in \mathbb{Q}^+$ such that $\vol(Z) > v_i$ for every $Z \subset X$  subvariety of dimension $i$ passing through a point $x \in \Pi$ and let $\mu_i:=\frac{i}{\sqrt[i]{v_i}}$. Set $$M:=\left[  \left( \frac{d}{\alpha} +1\right) \cdot (\mu_{d-1}+1) \cdot (\mu_{d-2}+1) \cdot \ldots \cdot (\mu_{1}+1) \right].$$ Then for all $n \geq 1$, for all $m \geq nM$, $h^0((m+1)K_X) \geq n$.

\end{theorem}

\begin{remark}
By \cite[1.4]{HcMcK} or \cite[1.2]{Takayama}, we know that there exists $\eta_i$ such that for every variety $Z$ of dimension $i$ and of general type, then $\vol(Z) \geq \eta_i$. Therefore for every $i=1, \ldots, d-1$, the $v_i$'s exist and are greater than $0$.
\end{remark}

\begin{proof}
As in \cite{Todorov} and as before, to prove the theorem we will essentially produce lc centres and then, using \cite[Theorem 4.1]{HcMcK}, cut their dimensions until they are points; then we can apply Nadel's vanishing theorem to pull back sections from the points to the variety. 

Let $X_0$ be the intersection between $\Pi$ and $X \setminus \mathbb{B}_+(K_X)$. $X_0$ is a very general subset of $X$, hence countably dense. Note that for every $x \in X_0$, every subvariety through $x$ is of general type, since its volume is strictly positive by hypothesis.

Since $\vol(K_X) > \alpha^d$, by \cite[2.2]{Todorov} (cf.\ also \cite[1.1.31]{LazI}), for every $x \in X$ and every $k \gg 0$ there exists a divisor $A_x \in |kK_X|$ with $\mult_x(A_x) > k \alpha$. Let $\Delta'_x:= A_x \frac{\lambda'_x}{k}$, with $\lambda'_x < \frac{d}{\alpha}$, $\lambda'_x \in \mathbb{Q}^+$, but close enough to $\frac{d}{\alpha}$ so that $\mult_x(\Delta'_x) > d$. Note that $\Delta'_x \sim \lambda_x'K_X$. Let $s_x:=lct(X, \Delta'_x, x)$. By \cite[9.3.2 and 9.3.12]{LazII}, $s_x < 1$ . Moreover, by \cite[9.3.16]{LazII}, $s_x \in \mathbb{Q}^+$. Therefore, without loss of generality, we can suppose that $(X, \Delta'_x)$ is lc, not klt in $x$.

By Lemma \ref{tiebreaking}, \ref{due}., for every $x \in X_0$ there exists an effective $\mathbb{Q}$-divisor $D_x \sim \lambda_x K_X$, with $\lambda_x < \frac{d}{\alpha}, \lambda_x \in \mathbb{Q}^+$, such that $(X, D_x)$ is lc, not klt in $x$ and $LLC(X, D_x, x)=\{V_x\}$, where $V_x$ is the unique minimal element of $LLC(X, \Delta'_x, x)$. Moreover we can also assume that $V_x$ is an exceptional lc centre.

For every $0 \leq i \leq d-1$, set $$Y_i:=\{x \in X_0 \textrm{ s.t. } \dim(V_x)=i\}.$$  
Since $X_0$ is countably dense then at least one between the $Y_i$'s is countably dense.  Moreover we can assume that $Y_{d-1}$ is countably dense - in fact, numerically, this is the ``worst'' possible scenario, as it will be clear further on in the proof. 

%in fact every time we want to cut down the dimension of a lc centre using \cite{HcMcK}, thm. 4.1, we need to produce new divisors $D$'s with greater $\lambda$'s.

Now we apply \cite[Theorem 4.1]{HcMcK}: for every $x \in Y_{d-1}$ consider $V_x$ and a resolution $f_x: W_x \rightarrow V_x$. As we have already seen, $V_x$ is an exceptional lc centre of $(X, D_x)$. Since $x \in X_0$, $V_x$, and hence $W_x$, are of general type and $V_x$ is not contained in the augmented base locus of $K_X$. Moreover $\vol(W_x) > v_{d-1}$ by hypothesis, since the volume is a birational invariant. Let $U_x $ be the very general subset of $V_x$ defined as in \cite[Theorem 4.1]{HcMcK}.  Set $U'_x:=U_x \cap X_0$. $U'_x$ is still a very general and non-empty subset of $V_x$. Moreover $\vol(\mu_{d-1} K_{W_x}) > (d-1)^{d-1}$. 
For every $y \in U'_x$ let us consider $y' \in f_x^{-1}(y) \subset W_x$. Since $y'$ is a smooth point, by \cite[1.1.31]{LazI} and \cite[ 9.3.2]{LazII}, there exists $\Theta_{y'} \sim \mu_{d-1} K_{W_x}$ such that $(W_x, \Theta_{y'})$ is not klt in $y'$. As before, since $lct(W_x, \Theta_{y'}, y') < 1$, we can suppose that $(W_x, \Theta_{y'})$ is lc, not klt in $y'$ and $\Theta_{y'} \sim \mu_{y'} K_{W_x}$ with $\mu_{y'} \in \mathbb{Q}^+$ and $\mu_{y'} < \mu_{d-1}$. Therefore there exists a pure lc centre $W'_{y'} \in LLC(W_x, \Theta_{y'},y')$. Set $V'_y:= f_x(W'_{y'}) \ni y $. By \cite[Theorem 4.1]{HcMcK}, for every $\delta \in \mathbb{Q}^+$ there exists a $\mathbb{Q}$-divisor $D'_y$ such that $V'_y$ is an exceptional lc centre and such that $D'_y \sim ((\lambda_x+1)(\mu_{y'}+1)-1+\delta)K_X$.  At the end we are in the following situation:  $\cup_{x \in Y_{d-1}}U'_x$ is countably dense in $X$ and for every $z \in \cup_{x \in Y_{d-1}}U'_x$ there exists a $\mathbb{Q}$-divisor $D'_z$ such that $LLC(X,D'_z,z)=\{V'_z\}$ with $V'_z$ exceptional lc centre, $\dim(V'_z)< d-1$ and $D'_z \sim_{\mathbb{Q}} ((\lambda_z+1)(\mu_z+1)-1+\delta)K_X$ with $\lambda_z < \frac{d}{\alpha}$ and $\mu_z < \mu_{d-1}$.

We can now apply \cite[Theorem 4.1]{HcMcK} again and again and conclude that there exists a countably dense set $\Gamma \subseteq X$ such that for every $x \in \Gamma$ there exists a $\mathbb{Q}$-divisor $B_x$ such that $LLC(X,B_x,x)=\{x\}$ and $B_x \sim_{\mathbb{Q}} \gamma K_X$ with $$\gamma < \left( \frac{d}{\alpha}+1\right) \cdot (\mu_{d-1}+1) \cdot (\mu_{d-2}+1) \cdot \ldots \cdot (\mu_{1}+1)-1+\delta q,$$
where $q$ is a positive rational number.

Taking $\delta$ sufficiently small, we can conclude that $\gamma < M$, therefore, by Lemma \ref{tanti divisori} and Lemma \ref{nonvanishing}, for all $n \geq 1$, for all $m \geq nM$, $h^0((m+1)K_X) \geq n$.
\end{proof}

\begin{remark} \label{dimensionealtarmk}
In the above proof it is clear that $$M \geq \left[ (\mu_{d-1}+1) \cdot (\mu_{d-2}+1) \cdot \ldots \cdot (\mu_{1}+1) \right]$$ and that ``$=$'' holds as soon as $$  \left( \frac{d}{\alpha}+1\right) \cdot (\mu_{d-1}+1)  \cdot \ldots \cdot (\mu_{1}+1)- \left[ (\mu_{d-1}+1) \cdot \ldots \cdot (\mu_{1}+1) \right]<1$$ i.e. $$\frac{d}{\alpha} < \frac{ 1-\{ (\mu_{d-1}+1)  \cdot \ldots \cdot (\mu_{1}+1) \}}{(\mu_{d-1}+1) \cdot \ldots \cdot (\mu_{1}+1)}$$ i.e. $$\alpha > \frac{d(\mu_{d-1}+1) \cdot \ldots \cdot (\mu_{1}+1)}{1-\{(\mu_{d-1}+1)  \cdot \ldots \cdot (\mu_{1}+1)\}}. $$
\end{remark}
\text{   }\\
%Let now $g$ be a positive integer and let $\Lambda$ be a very general subset of $X$ such that if $x \in \Lambda$ then every (possibly singular) curve through $x$ has geometric genus $\geq g$. As before $g \geq 2$.

\begin{corollary} \label{explicitnonvanishing}
Let $X$ be a smooth, not $g$-countably dense, projective variety of general type of dimension $d$ and such that $\vol(X) > \alpha^d$. If $d=3$, if $$ \alpha > \frac{9 \frac{2g-1}{2g-2}}{1-\left\{3 \frac{2g-1}{2g-2} \right\}}$$ then we have that $h^0\left(\left(1+m\right)K_X\right)\geq n $ for all $n \geq 1$ and all $m \geq \left[3\frac{2g-1}{2g-2}\right]n$. If $d=4$, $$ \alpha > \frac{12(3\sqrt[3]{2660}+1) \frac{2g-1}{2g-2}}{1-\left\{  3(3\sqrt[3]{2660}+1) \frac{2g-1}{2g-2}\right\}}$$ then $h^0\left(X, \left(1+m \right)K_X\right) \geq n$ for all $n \geq 1$ and all $m \geq  \left[3(3\sqrt[3]{2660}+1) \frac{2g-1}{2g-2}\right]n $. In general: if $d=3$, $\alpha > 27$ then $h^0(X, (1+m)K_X) \geq n$ for all $n \geq 1$ and all $m \geq 4n$; if $d=4$, $\alpha \geq 1709$ then $h^0(X, (1+m)K_X) \geq n$ for all $n \geq 1$ and all $m \geq 191n$.

\end{corollary}
\begin{proof}
For every $X$ and for every $0 < \epsilon \ll 1$ we can take $v_1=2g-2-\epsilon$ (by Remark  \ref{notcountablydense}), $v_2=1-\epsilon$ (the minimal model of a surface is nonsingular, hence the volume is an integer) and, by \cite{ChenChenII}, $v_3=\frac{1}{2660}-\epsilon$. Therefore $\mu_1=\frac{1}{2g-2}+o(1)$, $\mu_2=2+o(1)$ and $\mu_3=3\sqrt[3]{2660} +o(1)$ (with $o(1)>0$, $\lim_{\epsilon \rightarrow 0} o(1)=0$). 

If $X$ is a threefold we have that $ (\mu_{2}+1) \cdot (\mu_{1}+1) = 3\frac{2g-1}{2g-2}+o(1)$ therefore, by \ref{dimensionealta} and \ref{dimensionealtarmk}, if $$ \alpha > \frac{9 \frac{2g-1}{2g-2}}{1-\left\{3 \frac{2g-1}{2g-2} \right\}}$$ then $h^0\left(\left(1+m\right)K_X\right)\geq n$ for every $n\geq1$ and every $m \geq \left[3\frac{2g-1}{2g-2}\right]n$. In general, taking $g=2$ by Remark \ref{generaltype}, we can conclude that if $\alpha > 27$ then $h^0(X, (1+m)K_X) \geq n$ for all $n \geq 1$ and all $m \geq 4n$. 

If $X$ is a fourfold we have that $ (\mu_3+1) \cdot (\mu_{2}+1) \cdot (\mu_{1}+1) = 3(3\sqrt[3]{2660}+1)\frac{2g-1}{2g-2}+o(1)$, therefore we can conclude that if $$ \alpha > \frac{12(3\sqrt[3]{2660}+1) \frac{2g-1}{2g-2}}{1-\left\{  3(3\sqrt[3]{2660}+1) \frac{2g-1}{2g-2}\right\}}$$ then $h^0\left(X, \left(1+ m \right)K_X\right) \geq n$ for all $n \geq 1$ and all $m \geq \left[3(3\sqrt[3]{2660}+1) \frac{2g-1}{2g-2}\right]n$. In general, taking $g=2$, we can conclude that if $\alpha \geq 1709$ then $h^0(X, (1+m)K_X) \geq n$ for all $n \geq 1$ and all $m\geq 191n$.

\end{proof}

For the birationality of pluricanonical systems, using - as for the $3$-fold case - Takayama's result instead of Hacon--McKernan's, under the same notation and hypotheses of \ref{dimensionealta}, we can state that

\begin{theorem} \label{birationalitydimalta}
Let $X$ be a smooth projective variety of general type and of dimension $d$, such that $\vol(X) > \alpha^d$. Let $\Pi$ be a very general subset of $X$ and, for $i=1, \ldots, d-1$, let $v_i \in \mathbb{Q}^+$ such that $\vol(Z) > v_i$ for every $Z \subset X$  subvariety of dimension $i$ passing through a point $x \in \Pi$ and let $\mu_i:=\frac{i}{\sqrt[i]{v_i}}$. 
Setting, for every $i=1, \ldots, d-1$, $r_i:=\sqrt[i]{2} \mu_i$,  $$\overline{s}:= 2   \prod_{i=1}^{d-1} (1+r_i) -2,$$ $$\overline{t}:=  \sqrt[d]{2} d \prod_{i=1}^{d-1} (1+r_i),$$ we have that if $l \geq \left[\overline{s}+\frac{\overline{t}}{\alpha}\right]+2$ then the linear system $|l K_X|$ gives a birational map. 
\end{theorem}

\begin{proof}
As in the proof of Theorem \ref{birthreefold}, we can reduce ourselves to the following situation: for every $ 0 < \epsilon < 1$ there exists a smooth projective variety $X'$ and a birational morphism $\pi: X' \rightarrow X$ and a decomposition $\mu^*(K_X) \sim_{\mathbb{Q}} A+E$ where $A=A_{\epsilon}$ is an ample $\mathbb{Q}$-divisor and $E=E_\epsilon$ is an effective $\mathbb{Q}$-divisor. As in Theorem \ref{birthreefold} we will argue on $X'$. By \cite[Proposition 5.3]{Takayama}, we know that given two very general points $x,y \in X'$ there exists an effective $\mathbb{Q}$-divisor $D$ on $X'$ and a positive constant $a$ with $D \sim_\mathbb{Q} aA$ such that  $x,y \in Z(\mathcal{J}(X', D) )$ with $\dim Z(\mathcal{J}(X', D) )=0$ around $x$ or $y$, that is $x$ or $y$ is an isolated point of $Z(\mathcal{J}(X', D) )$. 
Besides, by the same proposition, we also know that  $a < s+t/ \sqrt[d]{\vol(X)} \leq s+t / \alpha$ where $s, t$ are non-negative constants defined as follows.  Let $s_i, s'_i, t_i$ ($i=1, \ldots, d$) be non-negative constants determined inductively as (cf.\ \cite[Notation 5.2]{Takayama}): $s_1= 0$, $t_1=\sqrt[d]{2} d /(1-\epsilon)$, $s'_i =s_i+\epsilon$, $$s_{i+1}= \left( 1 + \sqrt[d-i]{2} \frac{\mu_{d-i}}{1-\epsilon}  \right) s'_i + 2 \sqrt[d-i]{2} \frac{\mu_{d-i}}{1-\epsilon},$$ $$t_{i+1}=\left( 1 + \sqrt[d-i]{2} \frac{\mu_{d-i}}{1-\epsilon}  \right) t_i.$$
Finally, set $s:=s_d, t:=t_d$.

As in the proof of Theorem \ref{birthreefold}, we can say that, given $l \in \mathbb{N}$, $|l K_{X'}|$ separates two very general points in $X'$ as soon as $l \geq [a]+2$.

It can be easily seen that $s=\overline{s}+o(1)$ and $t=\overline{t}+o(1)$, with $o(1) > 0$ and such that $\lim_{\epsilon \rightarrow 0}o(1)=0$. Note that $\overline{s}$ and $\overline{t}$ do not depend on $\epsilon$.

Since $a < s + t/\alpha$ then $a < \overline{s}+\overline{t}/\alpha + o(1)$, therefore, taking $\epsilon$ sufficiently small, $[a] \leq \left[ \overline{s}+\overline{t}/\alpha\right]$ and thus we can conclude.

\end{proof}

\begin{remark}
In the above proof it is clear that $$\left[ \overline{s}+ \frac{\overline{t}}{\alpha} \right] \geq \left[ \overline{s}\right]$$ and that ``$=$'' holds as soon as  $$ \frac{\overline{t}}{\alpha} <1-\{\overline{s}\}$$ (where $\{\cdot\}$ is the fractional part), that is $$\alpha > \frac{\overline{t}}{1-\{\overline{s}\}}.$$
\end{remark}

We can now do explicit calculations in the case of fourfolds, using the same notation and estimates as in \ref{explicitnonvanishing}.

%nel seguente corollario non faccio conti espliciti per i threefolds, perch la birazionalit per i threefold, in questo caso generale, si ha per ogni mappa pluricanonica dalla 13 in poi. Ma questo risultato  peggiore di quanto gi trovato: cio 5K_X e' birazionale (per volume alto); siccome 2K_X e' non nullo (per volume alto) e 3K_X e' non nullo (per volume alto) allora nK_X  birazionale per ogni n \geq 7. In effetti per siccome con gli stessi sistemi si pu facilmente dimostrare che 6K_X e' birazionale (per volume alto) allora possiamo dire che nK_X e' birazionale per ogni n \geq 5.
\begin{corollary} \label{explicitbirationality}
Let $X$ be a smooth, not $g$-countably dense, projective fourfold of general type such that $\vol(X) > \alpha^4$. If $$\alpha >  \frac{4 \sqrt[4]{2} \left( \frac{g}{g-1} \right) (1+2\sqrt{2})(1+3 \sqrt[3]{5320})}{1-\left\{  2 \left( \frac{g}{g-1}\right)  (1+2 \sqrt{2})(1+ 3 \sqrt[3]{5320})     \right\}} $$ we have that the linear system $|l K_X|$ gives a birational map for every $$l \geq \left[  2 \left( \frac{g}{g-1}\right)  (1+2 \sqrt{2})(1+3 \sqrt[3]{5320})  \right].$$ In general, if $\alpha \geq 2816$ then $|l K_X|$ gives a birational map for every $l \geq 817$.

\end{corollary}

\begin{proof}
For every $X$ and every $0 < \epsilon \ll 1$, as in \ref{explicitnonvanishing} we can take $r_1=\frac{1}{g-1} + o(1)$, $r_2=2 \sqrt[2] {2}+ o(1)$, $r_3=3 \sqrt[3]{5320}+o(1)$. Therefore $$\overline{s}= 2 \left( \frac{g}{g-1}\right)  (1+2 \sqrt{2})(1+ 3 \sqrt[3]{5320})-2 + o(1)$$ and $$\overline{t}= 4 \sqrt[4]{2} \left( \frac{g}{g-1} \right) (1+2\sqrt{2})(1+3 \sqrt[3]{5320})+o(1).$$ Hence, by \ref{birationalitydimalta} and its remark, if $$ \alpha >  \frac{4 \sqrt[4]{2} \left( \frac{g}{g-1} \right) (1+2\sqrt{2})(1+3 \sqrt[3]{5320})}{1-\left\{  2 \left( \frac{g}{g-1}\right)  (1+2 \sqrt{2})(1+ 3 \sqrt[3]{5320})     \right\}} $$ then $|l K_X|$ gives a birational map for every $$l \geq \left[  2 \left( \frac{g}{g-1}\right)  (1+2 \sqrt{2})(1+3 \sqrt[3]{5320})  \right].$$ 

In general, taking $g=2$, we can conclude that if $\alpha \geq 2816$ then $|l K_X|$ gives a birational map for every $l \geq 817$.
\end{proof}

Going back over the proof of Theorem \ref{threefold} one realizes that even in the case of varieties of general type of dimension $d$ strictly greater than $3$ we can reach the dichotomy ``non-vanishing of pluricanonical system'' VS ``fibre space over a curve with fibres of small volume''. Unfortunately when $d > 3$ we are not able to get new information from the fibration. Anyway we can state the following theorem, whose proof is obtained merging together the proofs of Theorem \ref{threefold} and Theorem \ref{dimensionealta}. Note that since $i \leq d-2$ then this theorem can be made explicit also in the case of fivefolds.

\begin{theorem}  \label{fivefoldnonvanishing}
Let $X$ be a smooth projective variety of general type of dimension $d$ and such that $\vol(X) > \alpha^d$. Let $\Pi$ be a very general subset of $X$ and, for $i=1, \ldots, d-2$, let $v_i \in \mathbb{Q}^+$ such that $\vol(Z) > v_i$ for every $Z \subset X$  subvariety of dimension $i$ passing through a point $x \in \Pi$. Let $\mu_i:=\frac{i}{\sqrt[i]{v_i}}$ and $R:=\prod_{i=1}^{d-2} (\mu_i+1)$. Let $l$ be a positive integer, $l > R$. Let $\beta_1:=\frac{(d-1)R}{l-R}$, $\beta_2:=\frac{(d-1)(l+R)}{l-R}$. For all $\overline{\beta} > \beta_1$, setting $\tilde{\beta}:=\min\{\overline{\beta}, \beta_2\}$, if $$\alpha > \frac{d(1+(d-1)/ \tilde{\beta})R}{l-(1+(d-1)/\tilde{ \beta})R}$$ then either $h^0(lK_X) \geq 1$ (and for all $n \in \mathbb{N}^+$, $h^0(mK_X) \geq n$ for all $m \geq n(l-1)+1$) or $X$ is birational to a fibre space $X'$, with $f: X' \rightarrow B$, where $B$ is a curve, such that the volume of the general fibre is $\leq {\overline{\beta}}^{d-1}$.
\end{theorem}

Analogously for the birationality of pluricanonical maps:

\begin{theorem} \label{fivefoldbirationality}
Let $X$ be a smooth projective variety of general type of dimension $d$ and such that $\vol(X) > \alpha^d$. Let $\Pi$ be a very general subset of $X$ and, for $i=1, \ldots, d-2$, let $v_i \in \mathbb{Q}^+$ such that $\vol(Z) > v_i$ for every $Z \subset X$  subvariety of dimension $i$ passing through a point $x \in \Pi$. Let $\mu_i:=\frac{i}{\sqrt[i]{v_i}} $, $r_i:=\sqrt[i]{2} \mu_i$ and $P:=\prod_{i=1}^{d-2} (1+r_i)$. Let $l$ be a positive integer, $l > 2P-1$. Let $\beta_1:=\frac{2\sqrt[d-1]{2}(d-1)P}{l+1-2P}$, $\beta_2:=\frac{\sqrt[d-1]{2}(d-1)(l+1+4P)}{2(l+1-2P)}$. For all $\overline{\beta} > \beta_1$, setting $\tilde{\beta}:=\min\{\overline{\beta}, \beta_2\}$, if $$\alpha > \frac{d\sqrt[d]{2}(1+\sqrt[d-1]{2}(d-1)/ \tilde{\beta})P}{l+1-2(1+\sqrt[d-1]{2}(d-1)/\tilde{ \beta})P}$$ then either $|lK_X|$ gives a birational map or $X$ is birational to a fibre space $X''$, with $f: X'' \rightarrow B$, where $B$ is a curve, such that the volume of the general fibre is $\leq {\overline{\beta}}^{d-1}$.
\end{theorem}

\section{Acknowledgements}
This paper is part of my Ph.D.\ thesis. I would like to warmly thank my supervisor, prof. Angelo F. Lopez, for his great patience and helpfulness and for all the conversations we had. I would also like to thank G. Pacienza who first introduced me to Todorov's work.
\bibliographystyle{plain}
\bibliography{bibliography}

\end{document}